\documentclass[a4paper,12pt]{amsart}

\usepackage{amsmath}
\usepackage{amsthm}
\usepackage{amssymb}
\usepackage{stmaryrd}
\usepackage{MnSymbol}
\usepackage{bbm}
\usepackage{bm}
\usepackage{dsfont}

\usepackage[top=30truemm, bottom=30truemm, left=25truemm, right=25truemm]{geometry}

\numberwithin{equation}{section}

\newtheorem{theorem}{Theorem}[section]
\newtheorem{lemma}[theorem]{Lemma}
\newtheorem{corollary}[theorem]{Corollary}
\newtheorem{proposition}[theorem]{Proposition}

\newtheorem{definition}[theorem]{Definition}

\theoremstyle{definition}

\newcommand{\meas}{\mathrm{meas}}

\title[Effective hybrid joint universality]{Effective hybrid joint universality for Dirichlet $L$-functions and its application}
\author{Keita Nakai}
\date{}

\begin{document}

\begin{abstract}
    In 2003, Garunk\v{s}tis provided a lower bound for the lower density of the universality theorem for the Riemann zeta-function. In this paper, we generalize this result for the hybrid joint universality theorem for Dirichlet $L$-functions whose moduli are prime numbers. Furthermore, by its application, we estimate a lower bound of the lower density of the universality theorem for Hurwitz zeta-functions with rational parameters. 
\end{abstract}

\maketitle

\section{Introduction and main results} 
In this paper, we use the symbol $\Theta$ with a specific meaning: if the inequality
\[
|f(x)| \le c g(x)
\]
holds for some constant $c > 0$, then we write
\[
f(x) = \Theta(cg(x)).
\]
This notation follows that used in \cite{Ga03}. 
Note that unlike the standard Landau's big $O$-notation, we do not omit the constant $c$; instead, we include it explicitly to emphasize the dependence on $c$.

Let $s = \sigma + it$ be a complex variable. 
The Riemann zeta-function is defined by $\zeta(s) = \sum_{n=1}^{\infty} n^{-s}$ for $\sigma > 1$. 
The Riemann zeta-function can be continued meromorphically to the whole plane $\mathbb{C}$. 
In 1975, Voronin proved the approximation theorem called the universality theorem. 
Modern statement of universality is as follows. 
\begin{theorem} \label{Voronin}
Let $\mathcal{K}$ be a compact set in the strip $1/2 < \sigma < 1$ with connected complement, and let $f(s)$ be a non-vanishing continuous function on $\mathcal{K}$ that is analytic in the interior of $\mathcal{K}$. Then, for any $\varepsilon > 0$ 
\[
\liminf_{T \to \infty} \frac{1}{T} \meas \left\{\tau \in [0, T] :  \sup_{s \in \mathcal{K}} |\zeta(s + i\tau) -f(s)| < \varepsilon \right\} > 0,
\]
where $\meas$ denotes the 1-dimensional Lebesgue measure. 

\end{theorem}

Roughly speaking, any non-vanishing holomorphic function can be approximated by the Riemann zeta-function with some shift $\tau$, 
and the set of such $\tau$ has positive lower density. 
This universality theorem has been improved and extended in various zeta-functions and $L$-functions. 
One generalization is to combine universality and Kronecker-Weyl theorem, which is called hybrid joint universality now. 
First, Gonek proved the following hybrid joint universality theorem for Dirichlet $L$-functions. 
\begin{theorem} [{\cite[Theorem~3.1]{Gon}}]
    Let $q \ge 1$ be an integer and let $C$ be a simply connected compact set in $\frac{1}{2} < \sigma < 1$. 
    Suppose that for each prime $p|q$ we have $0 \le \theta_p < 1$ and that for each character $\chi \pmod q$, $f_\chi$ is continuous on $C$ and analytic in the interior of $C$. 
    If $\varepsilon > 0$, there is a $\tau \in \mathbb{R}$ 
    such that 
    \[
     \left\|\tau \frac{\log{p}}{2\pi} - \theta_p \right\| < \varepsilon
     \]
     for $p|q$ and 
     \[
     \max_{\chi \bmod q} \max_{s \in C} \left| L(s + i\tau, \chi) -e^{f_{\chi}(s)} \right| < \varepsilon, 
     \]
     which $\|x\|$ denotes the distance from $x$ to the nearest integer.
\end{theorem}
Later, Kaczorowski and Kulas slightly generalized this result and proved that the set of such $\tau$ has a positive density.
Moreover, hybrid joint universality has been studied and extended. 
Recently, in \cite{So}, hybrid joint universality was extended to the discrete setting, allowing shifts of the form $n^k$.
Furthermore, a probabilistic proof of hybrid joint universality was established by Endo~\cite{En24}.
We can see other developments of hybrid joint universality in \cite{Ma}.

On the other hand, it is clear that we can decompose the universality theorem for the Riemann zeta-function into two parts as follows:
\begin{enumerate}
    \item[(i)] There exists $T > 0$ such that there is a $\tau \in [T, 2T]$ which the inequality 
    \[
    \sup_{s \in K} \left| \zeta(s + i\tau) - f(s) \right| < \varepsilon
    \]
    holds. 
    \item[(ii)] There exists a constant $C >0$ such that
    \[
    \liminf_{T \to \infty} \frac{1}{T} \meas \left\{\tau \in [T, 2T] : \sup_{s \in K} |\zeta(s + i\tau) -f(s)| < \varepsilon \right\} \ge C.
    \]
\end{enumerate}
It is a natural question of how small such $T$ is, and how large the positive lower density of the universality theorem is. 
The answer to the former question (i), and its generalization, were given in \cite{Ga03}, \cite{GLMS}, \cite{SS}, \cite{En}, and \cite{WSR}. 
The following is the most famous result related to (i) given by Garunk\v stis, Laurin\v cikas, Matsumoto, J. Stueding, and R. Steuding~\cite{GLMS}. 
\begin{theorem}[{\cite[Theorem~4]{GLMS}}]
    Let $s_0 = \sigma _0 + it_0$, $1/2 < \sigma_0 < 1$, $r > 0$, $K=\{s \in \mathbb{C} : |s-s_0| \le r \}$, and $g$ be analytic on $K$ with $g(s_0) \ne 0$. 
     Put $M(g) = \max_{|s-s_0| = r} |g(s)|$. Fix $0 < \varepsilon < 1$, and $0 < \delta_0 < 1$. 
    If $N = N(\delta_0, \varepsilon, g)$ and $T = T(g, \varepsilon, \sigma_0, \delta_0, N)$ satisfy 
   \[
  M(g) \frac{\delta_0^N}{1-\delta_0} < \frac{\varepsilon}{3}
\]
and
\[
T \ge \max\{c_0(\sigma_0, N) \exp(\exp\left(c_1(\sigma_0, N) A(N, g, (\varepsilon/3)\exp(\delta_0 r)))  \right), r\}
\]
then, there exists $\tau \in [T - t_0, 2T-t_0] $ such that  
  \[
\max_{|s-s_0| \le \delta r} |\zeta(s+i\tau) - g(s)| < \varepsilon
\]
for any $0 \le \delta < \delta_0$ satisfying
\[
 M(\tau) \frac{\delta^N}{1-\delta} < \frac{\varepsilon}{3}. 
\]
 Here, $c_0(\sigma_0, N), c_1(\sigma_0, N)$ and $A(N, g, (\varepsilon/3)\exp(\delta_0 r))$ are effective constants, and $M(\tau) = \max_{|s-s_0| =r} |\zeta(s + i\tau)|$. 
\end{theorem}
We note that Voronin~\cite{Vo} provided a quantitative version of the multidimensional denseness result for $\zeta(s)$.  

The answer to the latter question was first obtained in \cite{Ga03}. 
In 2003, Garunk\u{s}tis~\cite{Ga03} estimated a lower bound of the lower density under some restrictions. 
\begin{theorem}[{\cite[Corollary~2]{Ga03}}] \label{Ga eff con}
 Let $0 < \varepsilon < 1/2$ and $r = 0.0001$. 
       Let $g(s)$ be analytic in the disc $|s| \le 0.06$ with $\max_{|s| \le 0.06} |g(s)| \le 1$. 
       Then, 
       \begin{align*}
       &\liminf_{T \to \infty} \frac{1}{T} \meas \left\{\tau \in [T, 2T] : \sup_{|s| \le r} \left|\log \zeta\left(s + \frac{3}{4} + i\tau \right) -g(s) \right| < \varepsilon  \right\} \\ 
       &\ge \exp(-\varepsilon^{-13}). 
       \end{align*}
\end{theorem}
This result was inspired by Good's result \cite{Goo}.
After this work, Nakamura and Pankowski~\cite{NP} gave a lower bound related self-approximation, and Sourmelidis and Steuding~\cite{SS} also showed a lower bound of the weak version of universality for Hurwitz zeta-functions with algebraic irrational parameters under some conditions. 
Furthermore, in \cite{St03} and \cite{St05}, Steuding proved the non-trivial upper bound of the density of the universality theorem.
In 2018, Lamzouri, Lester, and Radziwi\l\l~\cite{LLR} proved another version of the effective universality theorem. 

In this paper, our aim is to estimate a lower bound of the  lower density of hybrid joint universality for Dirichlet $L$-functions. 
More precisely, we evaluate the following measure effectively when $q$ is a prime number.
\begin{theorem} [{\cite[Theorem~3.1]{Gon}, \cite[Theorem~3]{KK}}]
Let $q \ge 1$ be an integer and let $C$ be a connected compact set in $\frac{1}{2} < \sigma < 1$ with connected complement. 
    Suppose that for each prime $p|q$ we have $0 \le \theta_p < 1$ and that for each character $\chi \pmod q$, $f_\chi$ is continuous on $C$ and analytic in the interior of $C$. 
    Then, for $\varepsilon > 0$, 

     \begin{align*}
     \liminf_{T \to \infty} \frac{1}{T} \left\{\tau \in [0, T] : \max_{\chi \bmod q} \max_{s \in C} \left| L(s + i\tau, \chi) -e^{f_{\chi}(s)} \right| < \varepsilon \right. \\
     \left. \max_{p | q}\left\|\tau \frac{\log{p}}{2\pi} - \theta_p \right\| < \varepsilon \right\}
     >0
     \end{align*}
     
\end{theorem}
 
The following theorem is the main theorem of this paper using Garunk\v stis's way.  

\begin{theorem} \label{main1}
    Let $q$ be a prime number, $\chi_k$ be a Dirichlet character mod $q$ for $k=1, \dots, q-1$ and $\theta_q$ be a real number with $0 \le \theta_q < 1$. 
    Let $\varepsilon$, $\varepsilon_1$, $\beta$, $r$, and $R$ be positive numbers such that $\varepsilon \le 1$, $1.003\varepsilon_1 \le 1$, $0 < r < R < 1/4$, $0 < \beta + R < 1/4$ and $r < \delta e^{-1 - \frac{1}{4\delta}}$, where 
    \[
    \delta := \frac{\frac{1}{4} - R - \beta}{\log{\frac{e}{2R}}}. 
    \]
    Let $g_k(s)$ be analytic in $|s| \le R$ for $k = 1, \dots, q-1$, $\rho > 355991^{\frac{1}{1-4\delta}}$. Let 
    \[
    M := \max_{1 \le k \le q-1}\max_{|s| \le R} |g_k(s)| + 1.5 + \frac{3.42}{1-4R}
    \le \frac{\rho^\beta}{5e\delta^3 \log^4{\rho}} 
     \]
     and let $V, Q$ and $T$ satisfy the following conditions: 
     \[
     50 \le V \le \rho < Q, 
     \]
     $Q \ge \exp(q^8)$ and 
       \begin{align*}
        \log T &\ge \max\left\{\log{\pi} + \frac{1.02(0.75+r)(Q-872)}{0.25-r}, \right. \\ 
     &\left. \quad V\left(\frac{Q}{\rho} \right)^{\frac{1}{2}}(q-1)Q \left(241(q-1)Q + 434\left(\frac{1}{\varepsilon_1} -(q-1)\right) \log q + 217(q-2)\log\log q\right) \right\}. 
    \end{align*}
      Then 
     \[
     \alpha := \delta \log{\frac{\delta}{er}} - \frac{1}{4}
     \]
     is positive and the measure of $\tau \in [T, 2T]$, such that 
     \begin{align*}
         \max_{1 \le k \le q-1} \max_{|s| \le r} \left|\log{L\left(s + \frac{3}{4} + i\tau, \chi_k \right) - g_k(s)} \right|< \frac{M}{\left(\frac{R}{r} -1 \right) \rho^{\delta \log{\frac{R}{r}}}} + \frac{3}{\rho^\alpha \log{\rho}} +
         \frac{3\log{\rho}}{\rho^{(1-4\delta)(\frac{3}{4}-r)}} \\ + \frac{3}{\rho^{(1-4\delta)(\frac{1}{2}-2r)}\log{\rho}} + \varepsilon + 188 \frac{\rho^{\frac{1}{4} + r}}{V\log \rho} + \frac{3-4r}{1-4r} \frac{16}{\rho^{\frac{1}{4}-r} \sqrt{\log \rho}} 
     \end{align*}
     and 
     \[
     \left\|\tau \frac{\log{q}}{2\pi} - \theta_q \right\| < \frac{\varepsilon_1}{2}
     \]
     is greater than 
     \[
     \frac{T}{2} \left(\varepsilon_1 V^{-(q-1)\pi(\rho)} - 1.02(q-1) \frac{\varphi^4(q)}{q^4}\frac{\log^2{Q}}{\varepsilon^2 (0.25-r)^5 Q^{0.25-r}} \right),  
     \]
     where $\varphi$ is the Euler totient function. 
    \end{theorem}
Moreover, choosing suitable parameters, we have the next Corollary. 

\begin{corollary} \label{main2}
    Let $q$ be a prime number, $\chi_k$ be a Dirichlet character mod $q$ for $k=1, \dots, q-1$ and $\theta_q$ be a real number with $0 \le \theta_q < 1$. 
    Put $R = 0.06$, $\beta = 0.039$, $r \le 0.0001 $ in Theorem~\ref{main1}, and 
    \begin{align*}
        a(r) &:= \frac{(1-4\delta)^4\left(\frac{1}{4} - r\right)^4}{47920e\delta^3} + 3 + \frac{9}{2}(1-4\delta)\left(\frac{1}{4}-r \right) \\ 
        &+ 188\sqrt{\frac{(1-4\delta)(\frac{1}{4}-r)}{2}} + 16 \frac{3-4r}{1-4r}\sqrt{\frac{(1-4\delta)(\frac{1}{4}-r)}{2}}.  
    \end{align*} 
    Let $g_k(s)$ be analytic in the disc $|s| \le r$ for $k=1, \dots, q-1$. 
    Let $\varepsilon$, $\varepsilon_1$, $\rho$ be positive numbers such that $\varepsilon \le 1$, $2.006\varepsilon_1 \le 1$, 
    \[
    \rho \ge \left(\frac{2a(r)}{\varepsilon} \right)^{\frac{2}{(1-4\delta)(\frac{1}{4}-r)}},\ \frac{\rho^{\beta}}{5e\delta^3 \log^4\rho} \ge \max_{1 \le k \le q-1} \max_{|s| \le 0.06} |g_k(s)| + 6. 
    \]
    Then, 
    \begin{align*}
        &\liminf_{T \to \infty} \frac{1}{T} \meas\left\{\tau \in [T, 2T] : \max_{1 \le k \le q-1} \max_{|s| \le r} \left|\log{L\left(s + \frac{3}{4} + i\tau, \chi_k \right) - g_k(s)} \right|< \varepsilon, \right. \\ 
      &\hspace{5cm} \left. \left\|\tau \frac{\log{q}}{2\pi} - \theta_q \right\| < \varepsilon_1 \right\}  \ge \varepsilon_1 e^{-(q-1)\rho}  >0.
    \end{align*}
\end{corollary}

Actually, this formulation is a bit different from Garunk\v stis's original statement. 
This formulation is based on the result \cite{NP} of Nakamura and Pa\'nkowski. 

\section{Approximation by trigonometric polynomials}

In this section, our purpose is to prove the following proposition. 

\begin{proposition} \label{prop1}
    Let $q$ be 1 or a prime number and $\chi$ be a Dirichlet character modulo $q$. 
    Let $\beta$, $r$ and $R$ be such that $0 < r < R < 1/4$, $0 < \beta + R < 1/4$ and $r < \delta e^{-1 -\frac{1}{4\delta}}$, where 
    \[
    \delta := \frac{\frac{1}{4} - R - \beta}{\log{\frac{e}{2R}}}. 
    \]
    Let $g(s)$ be analytic in $|s| \le R$ and let 
     \[
    M := \max_{|s| \le R} |g(s)| + 1.5 + \frac{3.42}{1-4R}
    \le \frac{\rho^\beta}{5e\delta^3 \log^4{\rho}}.
     \]
    Let $\rho > 355991^{\frac{1}{1-4\delta}}$. 
    Then, there are real numbers $\theta_p$, $p \le \rho$, $p \ne q$ such that 
    \begin{align*}
     g(s) &= - \sum_{p \le \rho} \log \left(1 - \frac{\chi(p) e^{-2\pi i\theta_p}}{p^{s + \frac{3}{4}}} \right) \\
         &+ \Theta\left(\frac{M}{\left(\frac{R}{r} -1 \right) \rho^{\delta \log{\frac{R}{r}}}} + \frac{3}{\rho^\alpha \log{\rho}} +
         \frac{3\log{\rho}}{\rho^{(1-4\delta)(\frac{3}{4}-r)}}  + \frac{3}{\rho^{(1-4\delta)(\frac{1}{2}-2r)}\log{\rho}} \right)
    \end{align*}
\end{proposition}

In order to prove this proposition, we introduce the following lemmas. 

\begin{lemma} \label{Du}

We have that 
\[
\pi(x) \le \frac{x}{\log{x}} \left( 1 + \frac{1}{\log{x}} + \frac{2.51}{\log^2x} \right) < 1.094 \frac{x}{\log{x}}
\]
for $x \ge 355991$. 
\end{lemma}

\begin{proof}
    This is obtained by Dusart~\cite{Du}. 
\end{proof}

\begin{lemma} \label{appro}
Let $q$ be 1 or a prime number and $\chi$ be a Dirichlet character modulo $q$. 
Let $\Delta_{\lambda\rho}$, $\lambda < \rho$, denote the set of vectors $z = (z_p)_{\lambda< p \le \rho}$ with complex components $z_p$ of modulus $\le 1$. 
Let $g_k$, $k =0, 1, \dots$ be functions on $\Delta_{\lambda\rho}$ defined by 
\[
g_k(z) = \sum_{\lambda < p \le \rho} \chi(p) z_p p^{-\sigma} (-\log{p})^k,\ \ \ \sigma \le 1 - \varepsilon. 
\]
Let $K$ be a positive integer and $w_k$, $k = 0, \dots, K$,  complex numbers such that 
\[
K^3 \le 0.06 \log^2\lambda \log\frac{\rho}{\lambda}
\]
and 
\[
|w_k| \le \frac{\lambda^{1-\sigma}\log{\rho}}{10K^3\log{\lambda}} \left(\frac{1-\frac{\log{\lambda}}{{\log{\rho}}}}{2K} \right)^{K + 1} k! (K-k)! \log^k\rho.
\]
Then, the system of equations
\[
g_k(z) = w_k,\quad k = 0, \dots, K,
\]
has a solution $z$ in $\Delta_{\lambda\rho}$ for $\lambda \ge 355991$. 
\end{lemma}

\begin{proof}
This lemma is an analogue of \cite[Lemma~5]{Ga03} and \cite[Lemma~9]{Goo}. 
We can prove this lemma by the same way as \cite[Lemma~9]{Goo}.  
\end{proof}

\begin{lemma} [{\cite[Lemma~6]{Goo}}] \label{lemma of Good}
    Let $K$ and $L$ be positive integers and $K \le L$. 
    Let $a_{kl}$ and $b_k$, $1 \le k \le K$, $1 \le l \le L$, denote complex numbers. 
    Suppose that system of equations 
    \[
    \sum_{l=1}^{L} a_{kl}z_l = b_k,\ 1 \le k \le K, 
    \]
    has a solution $(z_1, \dots, z_L)$ belonging to 
    \[
    \Delta^L = \{(z_1, \dots, z_L) : \text{$z_l$ complex and $|z_l| \le 1$ for $1 \le l \le L$}\}. 
    \] 
    Then the above equation has a solution $(z'_1, \dots, z'_L)$ in $\Delta^L$ such that $|z'_l| = 1$ for at least $L - K$ positive integers $l \le L$. 
\end{lemma}

\begin{lemma} \label{expansion}
    Let $q$ be a positive integer and $\chi$ be a Dirichlet character modulo $q$. 
    Let $\sigma > 1/2$, $355991 \le\lambda < \rho$ and $\theta_p \in \mathbb{R}$. 
    Then
    \[
    -\sum_{\lambda < p \le \rho} \log \left(1 - \frac{\chi(p)e^{-2\pi i \theta_p}}{p^s} \right) = \sum_{\lambda < p \le \rho} \frac{\chi(p)e^{-2\pi i \theta_p}}{p^s} + \Theta \left(\frac{1.11 \sigma}{2\sigma -1} \frac{1}{\lambda^{2\sigma -1}\log{\lambda}} \right). 
    \]
\end{lemma}

\begin{proof}
We calculate the left hand side. 
By the Taylor expansion, we have 
\begin{align*}
     -\sum_{\lambda < p \le \rho} \log \left(1 - \frac{\chi(p)e^{-2\pi i \theta_p}}{p^s} \right) &= \sum_{\lambda < p \le \rho} \sum_{k=1}^{\infty} \frac{1}{k} \frac{\chi^{k}(p)e^{-2 k \pi i \theta_p}} {p^{ks}} \\
     &= \sum_{\lambda < p \le \rho} \frac{\chi(p)e^{-2 \pi i \theta_p}} {p^{s}} + \sum_{\lambda < p \le \rho} \sum_{k=2}^{\infty} \frac{1}{k} \frac{\chi^{k}(p)e^{-2 k \pi i \theta_p}} {p^{ks}}. 
\end{align*}
    Applying the partial summation and Lemma~\ref{Du}, we obtain 
   \begin{align*}
       \left| \sum_{\lambda < p \le \rho} \sum_{k=2}^{\infty} \frac{1}{k} \frac{\chi^{k}(p)e^{-2 k \pi i \theta_p}} {p^{ks}}\right| \le \frac{1.11\sigma}{2\sigma -1} \frac{1}{\lambda^{2\sigma -1}\log{\lambda}}. 
   \end{align*}
    
\end{proof}

Using these lemmas, we prove Proposition~\ref{prop1}. 

\begin{proof} [Proof of Proposition~\ref{prop1}]
 Let $f(s)$ be analytic in $|s| \le R$. By the Taylor expansion, for $|s| \le r$, 
 \begin{align} \label{taylor}
 f(s) = \sum_{k=0}^{K} a_k s^k + \Theta \left( \frac{(\frac{r}{R})^K \max_{|u| \le R} |f(u)|}{\frac{R}{r} -1} \right). 
 \end{align}
 In particular, the coefficients $a_k$ satisfy
 \begin{align} \label{bound}
      |a_k| \le \frac{\max_{|u| \le R} |f(u)|}{R^k}
 \end{align}
 for $k = 0, \dots, K$. 
 
First, we approximate the above polynomial by  a trigonometric polynomial 
\[\sum_{\lambda < p \le \rho} \chi(p) e^{-2\pi i \theta_p} p^{-s -\frac{3}{4}}. 
\] 
We put 
\[
K := \lfloor \delta \log{\rho} \rfloor \ \text{and}\ \lambda :=\rho e^{-4K} \ge \rho^{1 - 4\delta}. 
\]
 We note that $K$ satisfies the assumption of Lemma~\ref{appro}.  

 Using the same way as \cite[Proposition~1]{Ga03}, we have 
 \begin{align*} 
     & \left| \sum_{\lambda < p \le \rho} \chi(p) z_p p^{-s -\frac{3}{4}} - \sum_{k=0}^{K}s^k \sum_{\lambda < p \le \rho} \chi(p) z_p p^{-\frac{3}{4}} \frac{(-\log p)^k}{k!}  \right| \le \frac{3}{\rho^\alpha \log{\rho}}. 
 \end{align*}
 If $\sigma = \frac{3}{4}$ and complex numbers $\{w_k\}_{0 \le k \le K}$ satisfy 
 \[
 |w_k| \le \frac{1}{5e} \rho^{\frac{1}{4}} (\log{\rho})^{k-K-1} K^{-3} (K - k)! \left(\frac{e}{2}\right)^{-K}
 \]
 then, we have
 \begin{align*}
      &\frac{\lambda^{1-\sigma}\log{\rho}}{10K^3\log{\lambda}} \left(\frac{1-\frac{\log{\lambda}}{{\log{\rho}}}}{2K} \right)^{K + 1} k! (K-k)! \log^k\rho \\
      &\ge \frac{1}{5e} \rho^{\frac{1}{4}} (\log{\rho})^{k-K-1} K^{-3} (K-k)! \left(\frac{e}{2}\right)^{-K} \ge |w_k|.  
 \end{align*}
 Therefore, with Lemma~\ref{appro}, we find complex numbers $z_p$ satisfying $|z_p| \le 1$ and 
 \begin{align} \label{equ1}
     \sum_{k=0}^{K} s^k \sum_{\lambda < p \le \rho} \chi(p) z_p p^{-\frac{3}{4}} \frac{(-\log{p})^k}{k!} = \sum_{0 \le k \le K} s^k w_k. 
 \end{align}
 Putting $\beta = \frac{1}{4} - R - \delta \log \frac{e}{2R}$, we obtain 
 \begin{align*}
     \rho^{\frac{1}{4}} (K - k)! (\log\rho)^{k-K} \left( \frac{e}{2R} \right)^{-K} \ge \frac{\rho^\beta}{R^k}
 \end{align*}
 Thus, (\ref{equ1}) holds if 
 \[
 |w_k| \le \frac{\rho^\beta}{5e\delta^3 R^k \log^4 \rho}. 
 \]
On the other hand, if 
\[
\max_{|s| \le u} |f(u)| \le \frac{\rho^\beta}{5e\delta^3 \log^4 \rho}
\]
 holds, then coefficients $a_k$ satisfy
 \[
 |a_k| \le \frac{\max_{|u| \le R} |f(u)|}{R^k} \le \frac{\rho^\beta}{5\delta^3 R^k \log^4 \rho}. 
 \]
 Hence, by the same arguments, there are $|z_p| \le 1$ for $\lambda < p \le \rho$ such that 
 \begin{align} \label{approximation eff}
 f(s) = \sum_{\lambda < p \le \rho} \frac{\chi(p) z_p}{p^{s + \frac{3}{4}}} + \Theta \left( \frac{3}{\rho^\alpha \log{\rho}} + \frac{\max_{|u| \le R} |f(u)|} {\left(\frac{R}{r} - 1\right) \rho^{\delta \log{\frac{R}{r}}}} \right). 
 \end{align}
 Next, we replace these $z_p$ by complex numbers of modulus $= 1$. 

 For $|s| \le r$, if $\eta$ is greater than 355991 and $N$ is the integer part of $0.3\log{\eta}$, then we have 
 \begin{align*}
 \sum_{\eta < p \le 2\eta} \frac{\chi(p) z_p}{p^{s + \frac{3}{4}}} &= \sum_{0 \le k \le N } s^k \sum_{\eta < p \le 2\eta} \frac{\chi(p) z_p}{p^{\frac{3}{4}}} \frac{(-\log{p})^k}{k!}  +  \sum_{k > N} s^k \sum_{\eta < p \le 2\eta} \frac{\chi(p) z_p}{p^{\frac{3}{4}}} \frac{(-\log{p})^k}{k!} \\
 &=\sum_{0 \le k \le N } s^k \sum_{\eta < p \le 2\eta} \frac{\chi(p) z_p}{p^{\frac{3}{4}}} \frac{(-\log{p})^k}{k!} + \Theta(0.3 \eta^{-\frac{3}{4}}). 
 \end{align*}
 Here, we apply Lemma~\ref{lemma of Good} to 
 \[
 b_k = \sum_{\eta < p \le 2\eta} \frac{\chi(p) (-\log{p})^k}{k!} z_p,\ k=0, \dots, N. 
 \]
 Then, there exist complex numbers $z'_p$ with $|z'_p| \le 1$, $\eta < p \le 2\eta$, and $|z'_p| = 1$ for at least $\pi(2\eta) - \pi(\eta) - N$ primes $p$ such that 
 \[
 \sum_{\eta < p \le 2\eta} \frac{\chi(p) (-\log{p})^k}{k!} z_p = \sum_{\eta < p \le 2\eta} \frac{\chi(p) (-\log{p})^k}{k!} z'_p,\ k=0, \dots, N.
 \]
 Therefore we obtain 
 \[
 \sum_{\eta < p \le 2\eta} \frac{\chi(p) z_p}{p^{s + \frac{3}{4}}} = \sum_{\eta < p \le 2\eta} \frac{\chi(p) z'_p}{p^{s + \frac{3}{4}}} + \Theta(0.6\eta^{-\frac{3}{4}}). 
 \]
 Hence we see that there exist $\theta_p \in \mathbb{R}$ $(\eta < p \le 2\eta)$ such that 
 \begin{align*}
     \sum_{\eta < p \le 2\eta} \frac{\chi(p)z_p}{p^{s + \frac{3}{4}}} = \sum_{\eta < p \le 2\eta} \frac{\chi(p) e^{-2\pi i \theta_p}}{p^{s + \frac{3}{4}}} + \Theta(N\eta^{r -\frac{3}{4}} + 0.6\eta^{-\frac{3}{4}})
 \end{align*}
 for $|s| \le r$.
 Accordingly, for any $z_p$ with $|z_p| \le 1$, $\lambda < p \le \rho$, there are $\theta_p \in \mathbb{R}$. $\lambda < p \le \rho$ such that 
 \begin{align} \label{eq1}
     \sum_{\lambda < p \le \rho} \frac{\chi(p)z_p}{p^{s + \frac{3}{4}}} &= \sum_{0 \le j \le N_0} \sum_{2^j \lambda < p \le 2^{j+1}\lambda} \frac{\chi(p)z_p}{p^{s + \frac{3}{4}}} \notag \\
     &= \sum_{0 \le j \le N_0} \sum_{2^j \lambda < p \le 2^{j+1}\lambda} \frac{\chi(p) e^{-2\pi i \theta_p}}{p^{s + \frac{3}{4}}} + \Theta(\sum_{0 \le j < N_0} (\lfloor 0.3\log{2^j \lambda} \rfloor (2^j \lambda)^{r -\frac{3}{4}} + 0.6 (2^j\lambda)^{-\frac{3}{4}})) \notag \\
     &= \sum_{\lambda < p \le \rho} \frac{\chi(p) e^{-2\pi i \theta_p}}{p^{s + \frac{3}{4}}} + \Theta(1.1 \lambda^{r -\frac{3}{4}} \log\lambda), 
 \end{align}
 where $N_0 = (\log(\rho/\lambda))/\log{2} -1$. 
 Therefore we can replace $z_p$ in (\ref{approximation eff}) by complex numbers of modulus $= 1$. 

Next, we consider the following summation
\[
\sum_{\substack{p_n \le \lambda \\ p_n \ne q}} \log \left( 1 - \frac{(-1)^n}{p_n^{s + \frac{3}{4}}} \right). 
\]
Using the Taylor expansion, we obtain 
 \begin{align} \label{ineq1}
    \left| \sum_{\substack{p_n \le \lambda \\ p_n \ne q}} \log \left( 1 - \frac{(-1)^n}{p_n^{s + \frac{3}{4}}} \right) \right| 
    \le \left| \sum_{\substack{p_n \le \lambda \\ p_n \ne q}} \frac{(-1)^n}{p_n^{s + \frac{3}{4}}} \right| 
    + \sum_{p_n \le \lambda} \sum_{k=2}^{\infty} \frac{1}{k p_n^{k (\sigma + \frac{3}{4})}}
 \end{align}
 for $|s| \le r$. 
 We define $A_q(x) := \sum_{p \le x} (-1)^{\pi(p)} \chi_0(p)$, where $\chi_0$ is the principal character modulo $q$. 
 Then, the first term of the left hand side in (\ref{ineq1}) can be estimated as follows:
 \begin{align*}
\left| \sum_{\substack{p_n \le \lambda \\ p_n \ne q}} \frac{(-1)^n}{p_n^{s + \frac{3}{4}}} \right| 
 &= \left| \int_{2^-}^{\lambda} x^{-s -\frac{3}{4}}\,  dA_q(x)\right| \\
 &\le \frac{|A_q(\lambda)|}{\lambda^{\sigma + \frac{3}{4}}}+ \left| s + \frac{3}{4} \right| \int_{2}^{\lambda} |A_q(x)| x^{-\sigma - \frac{7}{4}}\, dx. 
 \end{align*}
 Since $|A_q(x)| \le 2$ for all $x$, we obtain 
 \begin{align*}
 \left| \sum_{\substack{p_n \le \lambda \\ p_n \ne q}} \frac{(-1)^n}{p_n^{s + \frac{3}{4}}} \right| 
 &\le  \frac{\left| s + \frac{3}{4} \right|} {\sigma + \frac{3}{4}} 2^{r  + \frac{1}{4}} \le 1.5. 
\end{align*}
By direct calculations, we can estimate the second term of the left hand side in (\ref{ineq1}) as
\[
\sum_{p_n \le \lambda} \sum_{k=2}^{\infty} \frac{1}{k p_n^{k (\sigma + \frac{3}{4})}} \le \frac{3.42}{1-4R}.  
\]
Therefore, 
\begin{align} \label{inequ2}
\left|\sum_{\substack{p_n \le \lambda \\ p_n \ne q}} \log \left( 1 - \frac{(-1)^n}{p_n^{s + \frac{3}{4}}} \right) \right| \le 1.5 + \frac{3.42}{1-4R}. 
\end{align}

 Finally, we apply the above arguments to 
 \begin{align*}
 f(s) = g(s) + \sum_{p_n \le \lambda} \log \left( 1 - \frac{\chi(p_n) e^{-2\pi i \theta_{p_n}}}{p_n^{s + \frac{3}{4}}} \right), 
 \end{align*}
 where $g(s)$ satisfies assumptions of Proposition~\ref{prop1} and for $p_n \ne q$, $\theta_{p_n}$ is a real number with $\chi(p_n) e^{-2\pi i \theta_{p_n}} = (-1)^n$.  
 By (\ref{approximation eff}), (\ref{eq1}), (\ref{inequ2}) and Lemma~\ref{expansion}, we obtain the conclusion. 

\end{proof}

\section{Approximation by a finite product}

Let 
\[
L_Q(s, \chi) = \prod_{p \le Q} \left(1 -\frac{\chi(p)}{p^s} \right)^{-1}, 
\]
where $Q \in \mathbb{N}$ and $s \in \mathbb{C}$. 

The purpose of this section is to prove the next proposition. 

\begin{proposition} \label{prop2}
    Let $q_1, \dots,  q_K$ be 1 or prime numbers with $q_1 \le \dots \le q_K$ and $\chi_k$ be a Dirichlet character modulo $q_k$ for $k = 1, \dots, K$. 
    Let $0 < r < \frac{1}{4}$ and $0 < \varepsilon \le 1$. 
    The measure of $\tau \in [T, 2T]$. such that 
    \[
    \max_{1 \le k \le K} \max_{|s| \le r} \left| \log L \left(s + \frac{3}{4} + i\tau, \chi_k \right) -  \log L_Q \left(s + \frac{3}{4} + i\tau, \chi_k \right)\right| < \varepsilon
    \]
    is greater than 
    \[
    T \left(1 - 0.51 K \frac{\varphi^4(q_K)}{q_K^4} \frac{\log^2 Q}{(0.25 - r)^5 Q^{0.25 - r} \varepsilon^2} \right)
    \]
    if  
    \[
    Q > \max\{\exp(q_K^8), 355991\} \ and\ T \ge \pi e^{\frac{1.02(0.75 + r)(Q -872)}{0.25 -r}}. 
    \]
\end{proposition}

We introduce lemmas to prove this proposition. 

\begin{lemma} [{\cite[Lemma~9]{Ga03}}] \label{exp}
Let $f(x)$ be a real-valued function on the interval $[a, b]$, and let $f'(x)$ be continuous and monotonic on $[a, b]$ and $|f'(x)| \le \delta < 1$. 
Then 
\[
\sum_{a < n \le b} e(f(n)) = \int_{a}^{b} e(f(x))\, dx + \Theta\left( \frac{4\sqrt{2} \delta}{\pi(1 - \delta)} + \frac{6\sqrt{2}\delta}{\pi} + 3 \right), 
\]
where $e(x) = e^{2\pi i x}$. 
\end{lemma}

\begin{lemma} \label{partial}
    Let $q \ge 1$ and $\chi$ be a Dirichlet character modulo $q$. 
    For $\sigma > 0$, $x \ge |t|/\pi$, $x \ge 1$, $s = \sigma + it$, 
    \begin{equation}
    L(s, \chi) = \sum_{n \le qx} \frac{\chi(n)}{n^s} + E_0(\chi) \frac{x^{1-s}}{s-1} + \Theta \left( \left(\varphi(q) \frac{7\sqrt{2}\pi^{-1} +4}{q^{\sigma}} + c_1(q)\right) x^{-\sigma} \right), 
    \end{equation}
    where 
    \begin{equation*}
  E_0(\chi)=
  \begin{cases}
    1 & \text{if $\chi = \chi_0$, } \\
    0 & \text{if $\chi \ne \chi_0$}, 
  \end{cases}
  c_1(q) = 
  \begin{cases}
      1 & \text{if $q = 1$,} \\
      \frac{\varphi(q)}{2} & \text{if $q \ge 2$.}
  \end{cases}
\end{equation*}

\end{lemma}

\begin{proof}
We can prove this lemma by the same way as \cite[Lemma~10]{Ga03} with Lemma~\ref{exp}. 
\end{proof}

We prepare the following lemma to prove Lemma~\ref{appro1} and Lemma~\ref{darith}. 

\begin{lemma} \label{arith}
For $q \in \mathbb{N}$ and $x \ge 1$, we have
\begin{equation} \label{arith eq}
\sum_{\substack{n \le x \\ (n, q) = 1}} \frac{1}{n} = \frac{\varphi(q)}{q} \left(\log x + \gamma + \sum_{p | q} \frac{\log p}{p-1}\right) + \Theta\left(\frac {2^{\nu(q)}}{x} \right), 
\end{equation}
where $\gamma$ is Euler's constant and $\nu(q)$ is the number of distinct primes dividing $q$.
\end{lemma}

\begin{proof}
    First we prove the case $q = 1$. 
    Using Euler's summation formula, we obtain 
    \[
    \sum_{n \le x} \frac{1}{n} =  \log x + 1 - \int_{1}^{\infty} \frac{t - [t]}{t^2}\, dt + \int_{x}^{\infty} \frac{t-\lfloor t \rfloor}{t^2}\, dt - \frac{x- \lfloor x \rfloor}{x}. 
    \]
    By simple observations, we have 
    \begin{align*}
     - \frac{1}{x} \le \int_{x}^{\infty} \frac{t- \lfloor t \rfloor}{t^2}\, dt - \frac{x-\lfloor x \rfloor}{x}  \le \frac{1}{x}\ \text{and}\ 
       \gamma = 1 - \int_{1}^{\infty} \frac{t - \lfloor t \rfloor}{t^2}\, dt. 
    \end{align*}
    Therefore, (\ref{arith eq}) holds when $q = 1$. 
    
    Next, we show the case $q \ge 2$ using the case $q = 1$. 
    By the standard method and the case $q = 1$, we have
    \begin{align*}
        \sum_{\substack{n \le x \\ (n, q) = 1}} \frac{1}{n} &= \sum_{d | q} \frac{\mu(d)}{d} \sum_{m \le \frac{x}{d}} \frac{1}{m} \\
        &= (\log x + \gamma) \sum_{d | q}  \frac{\mu(d)}{d} - \sum_{d | q} \frac{\mu(d) \log d}{d} + \Theta \left( \frac{2^{\nu(q)}}{x} \right).         
    \end{align*}
    By using
    \begin{align} \label{change1}
    \sum_{d | q} \frac{\mu(d)}{d} = \frac{\varphi(q)}{q},  
    \end{align}
    \begin{align} \label{change2}
    \sum_{d | q} \frac{\mu(d) \log d}{d} = -\frac{\varphi(q)}{q} \sum_{p|q} \frac{\log p}{p-1}, 
    \end{align}
    we obtain the Lemma. 
\end{proof}
Let $\tau(n)$ denote a number of divisors of $n$ including $1$ and $n$. 
\begin{lemma} \label{appro1}
Write 
\[
\sum_{n \le x} \tau(n) = x\log x + (2\gamma -1)x + \Delta(x),
\]
where $\Delta(x)$ is the error term for the divisor problem. 
Assume a bound 
\begin{equation} \label{divisor pro}
|\Delta(x)| \le C x^{\theta}\ (x \ge 1)
\end{equation}
with some $C  >0$ and $\theta > \frac{1}{4}$. 
For $q \in \mathbb{N}$ and $x \ge 3$, then we have
\begin{align*}
 S &:= \sum_{d|q} \mu(d) \sum_{\substack{n \le x \\ (n, q) = 1}} \left\{ \frac{x}{dn} \right\} \\
   &= \frac{\varphi^2(q)}{q^2} \left(1 -\gamma - \sum_{p|q} \frac{\log p}{p-1} \right) x + \Theta \left(C\sigma_{-\theta}^\flat(q)^2 x^\theta + \frac{\varphi(q)}{q} 2^{\nu(q)} \right), 
\end{align*}
where 
\[
\sigma_{\kappa}^\flat (q) := \sum_{d|q} \mu^2(d) d^\kappa.
 \]
\end{lemma}

\begin{proof}
    We first decompose the sum as 
    \[
    S = S_1 - S_2, 
    \]
    where 
    \begin{align*}
        S_1 := \sum_{d|q} \mu(d) \sum_{\substack{n \le x \\ (n, q) = 1}} \frac{x}{dn},\ 
        S_2 := \sum_{d|q}\mu(d) \sum_{\substack{n \le x \\ (n, q) = 1}} \left\lfloor \frac{x}{dn} \right\rfloor
    \end{align*}
    and evaluate these sums separately. 
    
    For the sum $S_1$, we have 
    \[
    S_1  =\frac{\varphi(q)}{q}x \sum_{\substack{n \le x \\ (n, q) = 1}} \frac{1}{n}. 
    \]
    By using Lemma~\ref{arith}, we get
    \begin{equation} \label{S1}
        S_1 = \frac{\varphi^2(q)}{q^2} \left(\log x + \gamma + \sum_{p|q} \frac{\log p}{p-1} \right)x + \Theta \left( \frac{\varphi(q)}{q} 2^{\nu(q)} \right). 
    \end{equation}
    For the sum $S_2$, note that if $n > x/d$, then $\left\lfloor \frac{x}{dn} \right\rfloor = 0$, and so 
    \[
    S_2 = \sum_{d|q}\mu(d) \sum_{\substack{n \le x/d \\ (n, q) = 1}} \left\lfloor \frac{x}{dn} \right\rfloor. 
    \]
    We then rewrite the sum by using 
    \[
    \mathbbm{1}_{(n, q)=1} = \sum_{\substack{e|q \\ em = n}} \mu(e)
    \]
    as 
    \begin{align*}
        S_2 &= \sum_{d, e|q} \mu(d) \mu(e) \sum_{m \le \frac{x}{de}} \left\lfloor \frac{x}{dem} \right\rfloor \\ 
        &= \sum_{d, e|q} \mu(d) \mu(e) \sum_{m \le \frac{x}{de}} \sum_{n \le \frac{x}{dem}} 1 \\
        &= \sum_{d, e|q} \mu(d) \mu(e) \sum_{mn \le \frac{x}{de}} 1 \\
        &= \sum_{d, e|q} \mu(d) \mu(e) \sum_{n \le \frac{x}{de}} \tau(n). 
    \end{align*}
    Using (\ref{change1}), (\ref{change2}) and the formula
    \[
    \sum_{n \le x} \tau(n) = x\log x + (2\gamma -1)x + \Delta(x), 
    \]
    we have
    \begin{align*}
        S_2 &= \sum_{d, e|q} \mu(d)\mu(e) \left( \frac{x}{de} \left(\log\frac{x}{de} + (2\gamma-1) \right) 
        + \Delta\left(\frac{x}{de} \right)\right) \\
        &= x(\log{x} + 2\gamma-1) \sum_{d, e|q} \frac{\mu(d)\mu(e)}{de} 
        -x \sum_{d, e|q} \frac{\mu(d)\mu(e)}{de} \log(de) 
        + \sum_{d, e|q} \mu(d)\mu(e) \Delta\left(\frac{x}{de} \right) \\
        &=\frac{\varphi^2(q)}{q^2}x(\log{x} + 2\gamma-1) 
        -2x\sum_{d|q} \frac{\mu(d)\log d}{d} \sum_{e|q} \frac{\mu(e)}{e} 
        + \sum_{d, e|q} \mu(d)\mu(e) \Delta\left(\frac{x}{de} \right) \\
        &= \frac{\varphi^2(q)}{q^2}x\left(\log{x} + 2\gamma-1 + 2\sum_{p|q}\frac{\log{p}}{p-1} \right) 
        + \sum_{d, e|q} \mu(d)\mu(e) \Delta\left(\frac{x}{de} \right). 
    \end{align*}
    By the assumption (\ref{divisor pro}), we obtain 
    \[
    \left|\sum_{d, e|q} \mu(d)\mu(e) \Delta\left(\frac{x}{de} \right) \right| \le Cx^\theta \left(\sum_{d|q} \mu^2(d) d^{-\theta} \right)^2 = C\sigma_{-\theta}^\flat(q) x^\theta. 
    \]
    Therefore we get
    \begin{align} \label{S2}
       S_2 =  \frac{\varphi^2(q)}{q^2}x\left(\log{x} + 2\gamma-1 + 2\sum_{p|q}\frac{\log{p}}{p-1} \right) + \Theta(C\sigma_{-\theta}^\flat(q) x^\theta). 
    \end{align}
   
    By combining (\ref{S1}) and (\ref{S2}), we obtain the assertion. 
\end{proof}

\begin{lemma} \label{darith}
For $q \in \mathbb{N}$, we have 
\begin{equation}
D(x, q) := \sum_{\substack{n \le x \\ (n, q) = 1}} \tau(n) = \frac{\varphi^2(q)}{q^2}x\log{x} + \frac{\varphi^2(q)}{q^2}c_2(q)x + \Theta\left(c_3(q)x^\theta + c_4(q)\right), 
\end{equation}
where 
\begin{align*}
    &c_2(q) = 2\gamma  +2\sum_{p|q} \frac{\log{p}}{p-1} -1, \\
    &c_3(q) = C\sigma_{-\theta}^\flat(q)^2 = C \left( \sum_{d|q} \mu(d)^2 d^{-\theta} \right)^2, \\
    &c_4(q) = 
    \begin{cases}
    0                 & \text{if $q=1$,} \\
    \frac{\varphi(q)}{q}2^{\nu(q) + 1} & \text{if $q > 1$,}
  \end{cases}
\end{align*}
and $\theta$ and $C$ are in (\ref{divisor pro}). 
Furthermore, we have 
\begin{align} \label{D2}
    D_2(x, q) &:= \sum_{\substack{n \le x \\ (n, q) = 1}} \tau^2(n) \\ & \le \frac{1}{6}\frac{\varphi^4(q)}{q^4}x\log^3 x + \frac{1}{2}\frac{\varphi^4(q)}{q^4} (1 + 2c_2(q))x\log^2 x \notag \\
    &\ + \left(2c_2(q)\frac{\varphi^2(q)}{q^2} + c_2^2(q)\frac{\varphi^2(q)}{q^2} + \frac{2}{1 - \theta}c_3(q)  + 2c_4(q) \right) \frac{\varphi^2(q)}{q^2}x\log{x} \notag \\
    &\ + \left(c_2^2(q) \frac{\varphi^2(q)}{q^2} + \frac{2}{1-\theta} c_2 (q)c_3(q) + 2c_2(q)c_4(q) -2\frac{\theta}{(1-\theta)^2}c_3^2(q) \right) \frac{\varphi^2(q)}{q^2}x \notag \\ 
    &\ + \theta c_3^2(q)x^\theta\log{x}  \notag \\ 
    &\ + \left( \frac{2\theta-1}{(1-\theta)^2}c_3(q) \frac{\varphi^2(q)}{q^2} - \frac{2\theta}{1-\theta}c_2(q)c_3(q)\frac{\varphi^2(q)}{q^2}  +c_3^2(q) + 2c_3(q)c_4(q) \right)x^\theta \notag \\
    &\ + \frac{1}{(1-\theta)^2}c_3(q)\frac{\varphi^2(q)}{q^2} + c_4^2(q). \notag
\end{align}
Particularly, if $q$ is 1 or a prime number and $x > \max\{\exp(q^8), 355991\}$, then 
\begin{equation} \label{darith2}
    D_2(x, q) \le 0.24 \frac{\varphi^4(q)}{q^4}x\log^3x. 
\end{equation}
\end{lemma}

\begin{proof}
First, we show (\ref{darith}). 
The case $q=1$ is obvious. 
We assume $q > 1$.
We note that 
\begin{align} \label{coprime}
    \sum_{\substack{n \le x \\ (n, q) =1}} 1 &= \sum_{d|q} \mu(d) \left\lfloor \frac{x}{d} \right\rfloor \notag \\
    &= \frac{\varphi(q)}{q}x - \sum_{d|q} \mu(d) \left\{ \frac{x}{d} \right\}.     
\end{align}
Using the above equation, we obtain 
\begin{align*}
    D(x, q) &= \sum_{\substack{nm \le x \\ (n, q) = (m, q) = 1}} 1 
    = \sum_{\substack{n \le x \\ (n,q) = 1}} \sum_{\substack{m \le \frac{x}{n} \\ (m, q) = 1}} 1 \\
    &= \sum_{\substack{n \le x \\ (n,q) = 1}} \left( \frac{\varphi(q)}{q} \frac{x}{n} - \sum_{d|q} \mu(d) \left\{ \frac{x}{dn} \right\}\right) \\
    &= \frac{\varphi(q)}{q}x \sum_{\substack{n \le x \\ (n,q) = 1}} \frac{1}{n} - \sum_{d|q} \mu(d) \sum_{\substack{n \le x \\ (n,q) = 1}} \left\{ \frac{x}{dn} \right\}. 
\end{align*}
Applying Lemma~\ref{arith} and Lemma~\ref{appro1}, we obtain the assertion. 

By the same manner as \cite[Lemma~12]{Ga03}, we can show (\ref{D2}). 
Finally, taking $\theta = \frac{1}{2}$ and $C = 0.397$ for $x \ge 5560$ (see \cite{BBR}), we can estimate $D_2(x, q)$ as (\ref{darith2}). 

\end{proof}

\begin{lemma} [{\cite[Lemma~13]{Ga03}}] \label{Dirichlet mean}
Let $a_1, \dots, a_n$ be arbitrary complex numbers. 
Then 
\[
\int_{\frac{T}{2}}^{T} \left|\sum_{n \le N} a_n n^{it} \right|^2\, dt = \frac{T}{2} \sum_{n \le N}|a_n|^2 + \Theta \left( 837 \sum_{n \le N} n|a_n|^2 \right). 
\]
\end{lemma}

\begin{lemma} \label{mean value eff}
Let $q$ be 1 or a prime number and $\chi$ be a Dirichlet character modulo $q$. 
Let $E = \{s: |s| \le d\} $ with $0 < d < 1/4$ and $\omega = 1/4 - d$. 
Then for 
\[
    Q > \max\{\exp(q^8), 355991\}\ and\ T \ge \pi e^{\frac{0.51(1-2\omega)(Q -872)}{\omega}}, 
\]
we have 
\begin{equation*}
    J := \int_{\frac{T}{2}}^{T}\iint_{E} \left| L\left(s + \frac{3}{4} + i\tau, \chi\right) L_Q^{-1}\left(s + \frac{3}{4} + i\tau, \chi\right) - 1\right|^2\, d\tau d\sigma dt \le 0.0062 T \frac{\varphi^4(q)}{q^4} \frac{\log^2{Q}}{\omega^3 Q^{2\omega}}. 
    \end{equation*}
\end{lemma}

\begin{proof}
By Lemma~\ref{partial}, we have 
\[
L\left(s + \frac{3}{4} + i\tau, \chi \right) = \sum_{n \le q\frac{T+d}{\pi}}\frac{\chi(n)}{n^{s+\frac{3}{4} + i\tau}} + \Theta\left(\left(7.8\frac{\varphi(q)}{q^\sigma} + c_1(q)\right) \left( \frac{\pi}{T} \right)^{\sigma + \frac{3}{4}} \right)
\]
for $s \in E$ and $\tau \in [T/2, T]$. 
Using the Minkowski inequality, we obtain the following:
\begin{align*}
    J &\le \iint_{E + \frac{3}{4}} \left( \int_{\frac{T}{2}}^{T} \left| L_Q^{-1}(s+i\tau, \chi)\sum_{n \le q\frac{T+d}{\pi}}\frac{\chi(n)}{n^{s + i\tau}} -1 \right|^2\, d\tau \right)d\sigma dt \\
    &\ +2 \left( \iint_{E + \frac{3}{4}} \int_{\frac{T}{2}}^{T} \left| L_Q^{-1}(s+i\tau, \chi)\sum_{n \le q\frac{T+d}{\pi}}\frac{\chi(n)}{n^{s + i\tau}} -1 \right|^2 d\tau d\sigma dt \right)^{\frac{1}{2}} \\ 
    &\ \ \times \left( \iint_{E + \frac{3}{4}} \left(7.8\frac{\varphi(q)}{q^{\sigma}} + c_1(q) \right)^2  \left(\frac{\pi}{T} \right)^{2\sigma} \int_{\frac{T}{2}}^T \left| L_Q^{-1}(s + i\tau, \chi) \right|^2 d\tau d\sigma dt \right)^{\frac{1}{2}} \\ 
    &\ \ +\iint_{E + \frac{3}{4}} \left(7.8\frac{\varphi(q)}{q^{\sigma}} + c_1(q) \right)^2  \left(\frac{\pi}{T} \right)^{2\sigma} \int_{\frac{T}{2}}^T \left| L_Q^{-1}(s + i\tau, \chi) \right|^2 d\tau d\sigma dt \\ 
    &=: J_1 + J_2 + J_3.
\end{align*}
From now, we calculate each terms. 
Combining bounds of $J_1$ and $J_3$, we can estimate $J_2$. 
Therefore it is sufficient to calculate $J_1$ and $J_3$. 
First, we estimate $J_1$. 
Putting $Q_1 = p_1 p_2 \dots p_{\pi(Q)}$, for $Q < q \frac{T+d}{\pi},$ we have
\[
L_Q^{-1}(s+i\tau, \chi)\sum_{n \le q\frac{T+d}{\pi}}\frac{\chi(n)}{n^{s + i\tau}} -1 = \sum_{Q < n < Q_1q\frac{T+d}{\pi}} \frac{\chi(n)b(n)}{n^{s+i\tau}}, 
\]
where $|b(n)| \le \tau(n)$. 
From the inequality $\sum_{p \le x} \log{p} \le 1.000081x$ for $x>0$ (see \cite{Sc}), we see that $Q_1 \le e^{1.000081Q}$. 
By Lemma~\ref{Dirichlet mean}, we have
\begin{align*}
&\int_{\frac{T}{2}}^{T} \left| L_Q^{-1}(s+i\tau, \chi)\sum_{n \le q\frac{T+d}{\pi}}\frac{\chi(n)}{n^{s + i\tau}} -1 \right|^2\, d\tau = \int_{\frac{T}{2}}^{T} \left| \sum_{Q < n < Q_1q\frac{T+d}{\pi}} \frac{\chi(n)b(n)}{n^{s+i\tau}} \right|^2\, d\tau \\
&\le \frac{T}{2} \sum_{\substack{Q < n < Q_1q\frac{T+d}{\pi} \\ (n, q) = 1}} \frac{\tau^2(n)}{n^{2\sigma}}
 + 837\sum_{\substack{Q < n < Q_1q\frac{T+d}{\pi} \\ (n, q) = 1}} \frac{\tau^2(n)}{n^{2\sigma}} \\
&\le \sigma T\int_{Q}^{Q_1 q\frac{T+d}{\pi}} \frac{D_2(u, q)}{u^{2\sigma}}\, du + 837(2\sigma-1)\int_{Q}^{Q_1q\frac{T+d}{\pi}} \frac{D_2(u, q)}{u^{2\sigma}}\, du + \frac{837 D_2(Q_1 q(T+d)/\pi, )}{(Q_1 q(\frac{T+d}{\pi}))^{2\sigma-1}}. 
\end{align*}
Applying Lemma~\ref{darith}, we can estimate
\begin{align*}
    \iint_{E+\frac{3}{4}}\sigma T\int_{Q}^{Q_1 q\frac{T+d}{\pi}} \frac{D_2(u, q)}{u^{2\sigma}}\, du 
    &\le 0.12\frac{\varphi^4(q)}{q^4}T\int_{\frac{1}{2}+\omega}^{1-\omega} \sigma \int_{Q}^{\infty} \frac{\log^3 u}{u^{2\sigma}}\, du \\ 
    &\le 0.12\frac{\varphi^4(q)}{q^4}T \frac{\log^2Q}{\omega^3Q^{2\omega}}\left( \frac{\omega^2}{2} + \frac{3\omega}{4\log{Q}} + \frac{7}{8\log^2{Q}} \right) \\
    &\le 0.006156\frac{\varphi^4(q)}{q^4}T \frac{\log^2Q}{\omega^3Q^{2\omega}}. 
\end{align*}
Next, by a simple calculation, we have
\begin{align*}
    &837\iint_{E+\frac{3}{4}} \left( (2\sigma-1) \int_{Q}^{Q_1q\frac{T+d}{\pi}} \frac{D_2(u, q)}{u^{2\sigma}}\, du + \frac{D_2(Q_1 q(T+d)/\pi, q)}{(Q_1 q(\frac{T+d}{\pi}))^{2\sigma-1}} \right)\, d\sigma dt \\ 
    &\le 200.88 \frac{\varphi^4(q)}{q^4} \log^3\left(Q_1q \frac{T+d}{\pi} \right) \iint_{E+\frac{3}{4}} \left( (2\sigma-1) \int_{Q}^{Q_1q\frac{T+d}{\pi}} \frac{du}{u^{2\sigma-1}} + \frac{1}{(Q_1 q(\frac{T+d}{\pi}))^{2\sigma-2}}\right)\, d\sigma dt \\
    &\le \frac{50.22}{\omega}\frac{\varphi^4(q)}{q^4} \log^3\left(Q_1q \frac{T+d}{\pi} \right) \int_{\frac{1}{2} + \omega}^{1-\omega}\left( Q_1q\frac{T+d}{\pi} \right)^{2-2\sigma}\, d\sigma \\
    &\le \frac{25.11}{\omega} \frac{\varphi^4(q)}{q^4} \left( Q_1q\frac{T+d}{\pi} \right)^{1-2\omega} \log^2\left(Q_1q \frac{T+d}{\pi} \right) \\
    &\le \frac{25.11}{\omega} \frac{\varphi^4(q)}{q^4}2^{1-2\omega}e^{1.01(1-2\omega)Q} \left(q \frac{T}{\pi} \right)^{1-2\omega} \left(1.01Q + \log{q\frac{T}{\pi}} + \log2 \right)^2.
\end{align*}
Now, $(qx)^{-\omega}(1.01Q + \log{qx} + \log2)$ is decreasing for $x \ge \frac{T}{\pi} > \exp\left(\frac{0.51(1-2\omega)(Q -872)}{\omega}  \right)$. 
Thus, we obtain 
\begin{align*}
&\frac{25.11}{\omega} \frac{\varphi^4(q)}{q^4}2^{1-2\omega}e^{1.01(1-2\omega)Q} \left(q \frac{T}{\pi} \right)^{1-2\omega} \left(1.01Q + \log{q\frac{T}{\pi}} + \log2 \right)^2 \\ 
&\le \frac{25.11}{\omega} \frac{\varphi^4(q)}{q^4}2^{\frac{1}{2}}e^{1.01(1-2\omega)Q} q \frac{T}{\pi} e^{-1.02(1-2\omega)(Q -872)} \left(1.01Q + \frac{0.51(1-2\omega)(Q-872)}{\omega} + \log2 \right)^2 \\
&\le \frac{25.11\cdot0.51^2\cdot2^{\frac{1}{2}}}{\omega^3}\frac{\varphi^4(q)}{q^3}Q^{-2\omega}\frac{T}{\pi}Q^{2+2\omega}e^{-0.01(1-2\omega)Q + 889.44(1-2\omega)} 
\end{align*}
By using the assumption $Q \ge \max\{\exp(q^8), 355991\}$, we see that  
\begin{align*}
    &qQ^{2+2\omega}e^{-0.01(1-2\omega)Q + 889.44(1-2\omega)} (\log Q)^{-2} \\
    &\le (\log{Q})^{\frac{1}{8}} Q^{2+2\omega}e^{-0.01(1-2\omega)Q + 889.44(1-2\omega)} (\log Q)^{-2} \le 6.818 \times 10^{-373}. 
\end{align*}
Hence, we have 
\begin{align*}
    &837\iint_{E+\frac{3}{4}} \left( (2\sigma-1) \int_{Q}^{Q_1q\frac{T+d}{\pi}} \frac{D_2(u, q)}{u^{2\sigma}}\, du + \frac{D_2(Q_1 q(T+d)/\pi, q)}{(Q_1 q(\frac{T+d}{\pi}))^{2\sigma-1}} \right)\, d\sigma dt\\
    &\le 2.1 \times 10^{-372} \frac{\varphi^4(q)}{q^4} T \frac{\log^2Q}{\omega^3 Q^{2\omega}}.  
\end{align*}
Therefore, $J_1$ can be estimated as 
\begin{align} \label{J_1}
    J_1 \le 0.006157 \frac{\varphi^4(q)}{q^4} T \frac{\log^2Q}{\omega^3 Q^{2\omega}}.
\end{align}

Next, we consider $J_3$. 
The finite product $L_Q^{-1}(s, \chi)$ can be represented by the Dirichlet polynomial
\[
L_Q^{-1}(s, \chi) = \sum_{n \le Q_1} \frac{c(n)}{n^s}, 
\]
where $|c(n)| \le 1$ and $c(n) = 0$ for $(n, q) >1$. 
Thus applying Lemma~\ref{Dirichlet mean}, we have 
\begin{align*}
   \int_{\frac{T}{2}}^T \left| L_Q^{-1}(s + i\tau, \chi) \right|^2\, d\tau &\le 
   \frac{T}{2} \sum_{n\le Q_1}\frac{|c(n)|^2}{n^{2\sigma}} + 837 \sum_{n \le Q_1} \frac{|c(n)|^2}{n^{2\sigma-1}} \\
   &\le \frac{T}{2} \sum_{\substack{n\le Q_1 \\ (n, q) = 1}}\frac{1}{n^{2\sigma}} + 837 \sum_{\substack{n \le Q_1 \\ (n, q) = 1}} \frac{1}{n^{2\sigma-1}}
\end{align*}
By the partial summation and (\ref{coprime}), for $\sigma \ne 1$ we get
\begin{align*}
    \sum_{\substack{n \le Q_1 \\ (n ,q) = 1}} \frac{1}{n^\sigma} &= \int_{1^{-}}^{Q_1} \frac{1}{u^\sigma}\, d\left(\sum_{\substack{n \le u \\ (n, q) = 1}} 1 \right) \\
    &=\frac{1}{Q_1^\sigma} \sum_{\substack{n \le Q_1 \\ (n, q) = 1}} 1 + \sigma \int_{1}^{Q_1} \left( \sum_{\substack{n \le u \\ (n, q) = 1}} 1\right) \frac{1}{u^{\sigma + 1}}\, du \\
    &\le \frac{\varphi(q)}{q} \frac{1}{Q_1^{\sigma-1}} +\frac{1}{Q_1^\sigma} + \frac{\varphi(q)}{q} \frac{\sigma}{1 - \sigma} \left(\frac{1}{Q_1^{\sigma-1}} -1\right) + 1-Q_1^{-\sigma} \\
    &= \frac{\varphi(q)}{q} \frac{1}{1 - \sigma} \frac{1}{Q_1^{\sigma-1}} + \frac{\varphi(q)}{q} \frac{\sigma}{\sigma - 1}+ 1 
\end{align*}
Therefore $J_3$ can be calculated as
\begin{align*}
    J_3 &\le \frac{1}{2} \left(7.8\frac{\varphi(q)}{q^{\frac{1}{2}}} + c_1(q)\right)^2 \int_{\frac{1}{2} + \omega}^{1-\omega} \left( \frac{\pi}{T}\right)^{2\sigma} \left(\frac{T}{2} \left( \frac{\varphi(q)}{q} \frac{1}{1 - 2\sigma} \frac{1}{Q_1^{2\sigma-1}} + \frac{\varphi(q)}{q} \frac{2\sigma}{2\sigma-1} + 1 \right) \right. \\ 
    &\qquad	\qquad \qquad \qquad \qquad \left. +837\left( \frac{\varphi(q)}{q} \frac{1}{2-2\sigma} \frac{1}{Q_1^{2\sigma-2}} + \frac{\varphi(q)}{q} \frac{2\sigma-1}{2\sigma-2}  + 1 \right) \right)\, d\sigma \\
    &\le \frac{1}{2} \left(7.8\frac{\varphi(q)}{q^{\frac{1}{2}}} + c_1(q)\right)^2 \int_{\frac{1}{2} + \omega}^{1-\omega} \left( \frac{\pi}{T}\right)^{2\sigma} \left(\frac{T}{2} \left( \frac{\varphi(q)}{q} \frac{2\sigma}{2\sigma-1} + 1 \right) \right. \\
    &\qquad	\qquad \qquad \qquad \qquad \left. +837 \left(\frac{\varphi(q)}{q} \frac{1}{2-2\sigma}\frac{1}{Q_1^{2\sigma-2}} + 1 \right) \right)\, d\sigma \\  
    &\le \frac{1}{2} \left(7.8\frac{\varphi(q)}{q^{\frac{1}{2}}} + c_1(q)\right)^2 \left( \frac{\varphi(q)}{q} \frac{\pi(2\omega + 1)}{8} + \frac{\pi}{4} \right. \\
    &\qquad	\qquad \qquad \qquad \qquad \left. + 837 \left( \frac{\varphi(q)}{q} \frac{\pi}{T} \frac{1}{4\omega \log Q_1} e^{1.01(1-2\omega)Q} + \frac{\pi}{2T}\right) \right) e^{-1.02(1-2\omega)(Q -872)}. 
\end{align*}
Now, by $c_1(q) \le \varphi(q)$, we estimate
\begin{align*}
    \left(7.8\frac{\varphi(q)}{q^{\frac{1}{2}}} + c_1(q)\right)^2 \le \varphi^2(q) \left(\frac{7.8^2}{q} + \frac{15.6}{q^{\frac{1}{2}}} + 1 \right) \le 77.44 \varphi^2(q). 
\end{align*}
Furthermore, using assumptions, we have 
\begin{align*}
&\left( \frac{\varphi(q)}{q} \frac{\pi(2\omega + 1)}{8} + \frac{\pi}{4} + 837 \left( \frac{\varphi(q)}{q} \frac{\pi}{T} \frac{1}{4\omega \log Q_1} e^{1.01(1-2\omega)Q} + \frac{\pi}{2T}\right) \right) e^{-1.02(1-2\omega)(Q -872)} \\
&\le \frac{6.3 \times 10^{-388}}{\omega Q^{2\omega}}
\end{align*}
Thus we obtain 
\begin{align} \label{J_3}
    J_3 \le \frac{2.44 \times 10^{-388}\varphi^2(q)}{\omega Q^{2\omega}} \le 3.3 \times 10^{-391} \frac{\varphi^4(q)}{q^4} \frac{\log^2Q}{\omega^3 Q^{2\omega}}. 
\end{align} 
Finally,  (\ref{J_1}) and (\ref{J_3}) implies the conclusion
\[
J \le 0.0062 T \frac{\varphi^4(q)}{q^4} \frac{\log^2{Q}}{\omega^3 Q^{2\omega}}.
\]

\end{proof}

\begin{lemma} [{cf. \cite[Appendix~Section~2~Theorem~7]{KV}}] \label{KV}
 If $f(z)$ is regular for $|z-z_0| \le R$ and 
 \[
 \iint_{|z-z_0| \le R} |f(z)|^2\, dxdy = H, 
 \]
then 
\[
|f(z)| \le \frac{(H/\pi)^{1/2}}{R-R'}
\]
for $|z-z_0| \le R' < R$. 
\end{lemma}
Using these lemmas, we prove Proposition~\ref{prop2}. 

\begin{proof} [Proof of Proposition~\ref{prop2}]
    Let $d = r + \frac{1/4 -r}{2}$ and $E = \{s : |s| \le d \}$. Then $\omega = \frac{1/4 -r}{2}$ and the set $A_T$ is defined by 
    \[
     A_T := \left\{\tau \in \left[\frac{T}{2}, T\right] : \iint_{E} \sum_{1 \le k \le K} \left| L\left(s + \frac{3}{4} + i\tau, \chi_k \right) L_Q^{-1}\left(s + \frac{3}{4} + i\tau, \chi_k\right) - 1\right|^2\, d\sigma dt \le (0.5\varepsilon\omega)^2\pi \right\}.  \] 
     By Lemma~\ref{mean value eff}, for $Q > \max\{\exp(q_K^8), 355991\}$ and $T \ge \pi e^{\frac{0.51(1-2\omega)(Q -872)}{\omega}}$, we have 
    \begin{align*}
        &\meas([T/2, T] \setminus A_T) \times (0.5\varepsilon\omega)^2\pi \\
        &< \int_{\frac{T}{2}}^{T} \iint_{E} \sum_{1 \le k \le K} \left| L\left(s + \frac{3}{4} + i\tau, \chi_k \right) L_Q^{-1}\left(s + \frac{3}{4} + i\tau, \chi_k\right) - 1\right|^2\, d\tau d\sigma dt \\ 
        &\le 0.0062 T K \frac{\varphi^4(q_K)}{q_K^4} \frac{\log^2{Q}}{\omega^3 Q^{2\omega}}.     \end{align*}
    Hence, we see that 
    \begin{align*}
    \meas(A_T) & \ge \frac{T}{2} \left(1 - 0.0124 K \frac{\varphi^4(q_K)}{q_K^4} \frac{\log^2{Q}}{\omega^3 Q^{2\omega}(0.5\varepsilon\omega)^2\pi}\right) \\
    &\ge \frac{T}{2} \left(1 - 0.51K \frac{\varphi^4(q)}{q^4} \frac{\log^2{Q}}{(0.25-r)^5Q^{0.25-r}\varepsilon^2}\right)
    \end{align*}
    and by Lemma~\ref{KV}, for $\tau \in A_T$, 
    \[
    \max_{1 \le k \le K} \max_{s \in E} \left| L\left(s + \frac{3}{4} + i\tau, \chi_k \right) L_Q^{-1}\left(s + \frac{3}{4} + i\tau, \chi_k \right) - 1\right| \le 0.5\varepsilon. 
    \]
    Therefore, for $\tau \in A_T$, we obtain 
    \[
\max_{1 \le k \le K} \max_{|s| \le r} \left| \log L \left(s + \frac{3}{4} + i\tau, \chi_k \right) -  \log L_Q \left(s + \frac{3}{4} + i\tau, \chi_k \right)\right| < \varepsilon.    \]
\end{proof}

\section{Application of Weyl's criterion}

In this section, our purpose is to prove the following proposition. 

\begin{proposition} \label{prop3}
    Let $q$ be a prime number and $\chi_k$ be a Dirichlet character mod $q$ for $k = 1, \dots, q-1$. 
    Let $\rho \ge 355991$, $50 \le V \le \rho < Q$, $q < Q$, 
    \[
   V\left(\frac{Q}{\rho} \right)^{\frac{1}{2}}(q-1)Q \left(241(q-1)Q + 434\left(\frac{1}{\varepsilon_1} -(q-1)\right) \log q  + 217(q-2)\log\log q\right) \le \log T,
    \]
    and $\lambda^{(k)}_p$, $p \le \rho$ be real numbers for $k=1, \dots, q-1$. 
    Let $r < \frac{1}{4}$, $\theta_q$ be a real number with $0 \le \theta_q < 1$ and $\varepsilon_1$ be a real number that satisfies $1.003\varepsilon_1 \le 1$. 
    Then the measure of $\tau \in [T, 2T]$, such that 
    \begin{equation*}
        \max_{1 \le k \le q-1} \max_{|s| \le r}\left|\sum_{p \le Q} \log\left(1 - \frac{\chi_k(p)}{p^{s + \frac{3}{4} + i\tau}} \right) - \sum_{p \le \rho} \log\left(1 - \frac{\chi_k(p)e^{-2\pi i \lambda^{(k)}_p}}{p^{s + \frac{3}{4}}} \right)\right| \le 188 \frac{\rho^{\frac{1}{4} + r}}{V\log \rho} + \frac{3-4r}{1-4r} \frac{16}{\rho^{\frac{1}{4}-r} \sqrt{\log \rho}} 
    \end{equation*}
    and 
    \begin{equation*}
      \left\|\tau \frac{\log q}{2\pi} - \theta_q \right\| < \frac{\varepsilon_1}{2}  
    \end{equation*}
    is greater than $\frac{1}{2}\varepsilon_1 TV^{-(q-1)\pi(\rho)}$. 
\end{proposition}

To prove this proposition, we prepare the following lemmas. 
The next lemma is a quantitative and continuous version of Weyl's criterion based on Koksma's lemma~\cite{Ko}. 

\begin{lemma} [{\cite[Lemma~18]{Ga03}}, cf. \cite{Ko}] \label{Koksma}
    Let $L$ be a positive integer and $a_1, \dots, a_L$, $b_1, \dots, b_L$ be real numbers such that $a_l < b_l \le a_l + 1$ for $l = 1, \dots, L$. 
    Let $U$ denote the set of points $x = (x_1, \dots, x_L)$ in $\mathbb{R}^L$, the $L$-dimentional vector space, satisfying $a_l \le x_l \le b_l$ for $l = 1, \dots, L$. 
    Let $\alpha = (\alpha_1, \dots, \alpha_L)$ be and $L$-tuple of real numbers and let $T_\alpha(U)$ denote the Jordan measure of the set $t \in (0, T)$, such that $t\alpha = (t\alpha_1, \dots, t\alpha_L) \in U$ mod 1. 
    Let $H_l$, $l = 1, \dots, L$, be greater than 1 and 
    \[
    \gamma_{h, l} = 
\begin{cases}
b_l - a_l + \frac{75}{H_l},  & \text{for $h = 0$, }\\
\min\left(\gamma_{0, l}, \frac{30}{|h|} \right),  & \text{for $h \ne 0$}. 
\end{cases}
    \]
    Then we have 
    \begin{align*}
        \left|\frac{T_\alpha(U)}{T} - \prod_{i=1}^{L} (b_l - a_l) \right| &\le \prod_{l=1}^{L}(b_l-a_l)\left( \prod_{l=1}^{L} \left(1 + \frac{75}{H_l(b_l-a_l)} \right) -1 \right) \\ 
        &\ + \frac{1}{T} \sideset{}{'}{\sum}_{(h)} \min \left( \left|\pi \sum_{l=1}^{L} \alpha_l h_l \right|^{-1}, T \right) \prod_{l=1}^{L} \gamma_{h_l. l}, 
    \end{align*}
    where the sum $\sideset{}{'}{\sum}_{(h)}$ extends over all $h = (h_1, \dots, h_L) \ne (0, \dots, 0)$ such that $h_l$ is an integer with $|h_l| \le H_l(1 + \min(\log H_l, \log L))$, $l = 1, \dots, L$. 
\end{lemma}

\begin{lemma} \label{appro2}
    Let $q$ be a positive integer, $\chi$ be a Dirichlet character modulo $q$ and $\lambda_p$, $\lambda'_p$ be real numbers.    
    Then, for $M > 355991$ and $|s| \le r < \frac{1}{4}$, 
    \begin{align*}
        \left| \sum_{p \le M} \log\left(1 -\frac{\chi(p) e^{2\pi i \lambda_p}}{p^{s + \frac{3}{4}}} \right) - \sum_{p \le M} \log \left(1 -\frac{\chi(p) e^{2\pi i\lambda'_p}}{p^{s + \frac{3}{4}}} \right) \right| \le 184 \frac{M^{\frac{1}{4} + r}}{\log{M}} \max_{p \le M}\|\lambda_p - \lambda'_p \| . 
    \end{align*}
\end{lemma}

\begin{proof}
    For $|s| \le r$, we have
    \begin{align*}
         &\left| \sum_{p \le M} \log\left(1 -\frac{\chi(p) e^{2\pi i\lambda_p}}{p^{s + \frac{3}{4}}} \right) - \sum_{p \le M} \log \left(1 -\frac{\chi(p) e^{2\pi i\lambda'_p}}{p^{s + \frac{3}{4}}} \right) \right| \\ 
         &\le \sum_{\substack{p \le M \\ (p, q) = 1}} \sum_{k=1}^{\infty} \frac{1}{k p^{k(\sigma 
 + \frac{3}{4})}} \left| e^{2\pi ik(\lambda_p - \lambda'_p)} -1 \right| 
    \le \left(2 + \sqrt{2}\right) \max_{p \le M} \left| e^{2\pi i(\lambda_p - \lambda'_p)} -1 \right| \sum_{\substack{p \le M}} \frac{1}{ p^{-r + \frac{3}{4}}}. 
    \end{align*}
    Using the fact that $\left| e^{2\pi ix} -1 \right| \le 2\pi\|x\|$ and 
    \begin{equation*}
        \sum_{p \le M} p^{r - \frac{3}{4}} \le 8.57 \frac{M^{r + \frac{1}{4}}}{\log{M}}
    \end{equation*}
    for $M > 355991$ (cf. \cite{NP}), we obtain the conclusion. 
\end{proof}

\begin{proof} [Proof of Proposition~\ref{prop3}]
Let $J = \left\lfloor V \left(\frac{Q}{\rho}\right)^{\frac{1}{4} + r} \right\rfloor$. 
When $q \le \rho$, we set
\begin{align*}
U(\theta)
:= \Bigl\{\, 
& x = \bigl(x_1, \dots, x_{(q-1)(\pi(\rho) - 1) + \pi(Q)- \pi(\rho) +1} \bigr) 
\\[2mm]
& :\ 
|x_l - \lambda_{p_m}^{(k)}| \le \frac{1}{2V},
      \ \text{$l = (k-1)(\pi(\rho) - 1) + m$ for $1 \le k \le q-1$ and $1 \le m < \pi(q)$},
\\[2mm]
&\ \ \ 
|x_l - \lambda_{p_m}^{(k)}| \le \frac{1}{2V},
      \ \text{$l = (k-1)(\pi(\rho) - 1) + m-1$ for $1 \le k \le q-1$ and $\pi(q) <  m \le \pi(\rho)$},
\\[2mm]
&\ \ \ 
  |x_l - \theta_p| \le \frac{1}{2J},
      \ \text{$l = (q-1)\pi(\rho) + m - \pi(\rho) -1$ for $\pi(\rho) < m \le \pi(Q)$,}
\\[2mm]
& \ \ \ |x_{(q-1)(\pi(\rho) - 1) + \pi(Q)- \pi(\rho) +1} - \theta_q| \le \frac{\varepsilon_1}{2}\Bigr\},
\end{align*}

\begin{align*}
     a_l =  
     \left\{
\begin{array}{ll}
\lambda^{(k)}_{p_{m}} - \frac{1}{2V}, & \text{$l = (k-1)(\pi(\rho) - 1) + m$ for $1 \le k \le q-1$ and $1 \le m  <\pi(q)$, }\\
\lambda^{(k)}_{p_{m}} - \frac{1}{2V}, & \text{$l = (k-1)(\pi(\rho) - 1) + m - 1$ for $1 \le k \le q-1$ and $\pi(q) < m  \le \pi(\rho)$, }\\
\theta_{p_m} - \frac{1}{2J}, & \text{$l = (q-1)(\pi(\rho) -1) + m - \pi(\rho)$ for $\pi(\rho) < m \le \pi(Q)$, } \\ 
\theta_{q} - \frac{\varepsilon_1}{2}, & \text{$l = (q-1)(\pi(\rho) - 1) + \pi(Q) - \pi(\rho)+ 1$, }
\end{array}
\right.
 \end{align*}
 \begin{align*}
     b_l =  
     \left\{
\begin{array}{ll}
\lambda^{(k)}_{p_{m}} + \frac{1}{2V}, & \text{$l = (k-1)(\pi(\rho) - 1) + m$ for $1 \le k \le q-1$ and $1 \le m < \pi(q)$, }\\
\lambda^{(k)}_{p_{m}} + \frac{1}{2V}, & \text{$l = (k-1)(\pi(\rho) - 1) + m - 1$ for  $1 \le k \le q-1$ and $\pi(q) < m  \le \pi(\rho)$, }\\
\theta_{p_m} + \frac{1}{2J},  & \text{$l = (q-1)(\pi(\rho) -1) + m - \pi(\rho)$ for $\pi(\rho) < m \le \pi(Q)$, } \\ 
\theta_{q} + \frac{\varepsilon_1}{2}, & \text{$l = (q-1)(\pi(\rho) - 1) + \pi(Q) - \pi(\rho)+ 1$, }
\end{array}
\right. 
 \end{align*}
 and 
  \begin{align*}
\alpha_l
&=\left\{
\begin{array}{ll}
\frac{1}{2\pi}\log p_m & \text{$l = (k-1)(\pi(\rho) - 1) + m$ for $1 \le k \le q-1$ and $1 \le m < \pi(q)$, } \\
\frac{1}{2\pi}\log p_m, & \text{$l = (k-1)(\pi(\rho) - 1) + m - 1$ for $1 \le k \le q-1$ and $\pi(q) < m  \le \pi(\rho)$, }\\
\frac{1}{2\pi}\log p_m, & \text{$l = (q-1)(\pi(\rho) -1) + m - \pi(\rho)$ for $\pi(\rho) < m \le \pi(Q)$, } \\ 
\frac{1}{2\pi}\log q, & \text{$l = (q-1)(\pi(\rho) - 1) + \pi(Q) - \pi(\rho)+ 1$. }
\end{array}
\right. 
\end{align*}
When $\rho < q < Q$, we set 
\begin{align*}
U(\theta)
:= \Bigl\{\, 
& x = \bigl(x_1, \dots, x_{(q-1)\pi(\rho) + \pi(Q)- \pi(\rho)} \bigr) 
\\[2mm]
& :\ 
|x_l - \lambda_{p_m}^{(k)}| \le \frac{1}{2V},
      \ \text{$l = (k-1)\pi(\rho) + m$ for $1 \le k \le q-1$ and $1 \le m \le \pi(\rho)$},
\\[2mm]
&\ \ \ 
|x_l - \theta_{p_m}| \le \frac{1}{2V},
      \ \text{$l = (q-1)\pi(\rho) + m - \pi(\rho)$ for  $\pi(\rho) <  m < \pi(q)$},
\\[2mm]
&\ \ \ 
  |x_l - \theta_{p_m}| \le \frac{1}{2J},
      \ \text{$l = (q-1)\pi(\rho) + m - \pi(\rho) -1$ for $\pi(q) < m \le \pi(Q)$,}
\\[2mm]
& \ \ \ |x_{(q-1)\pi(\rho) + \pi(Q)- \pi(\rho)} - \theta_q| \le \frac{\varepsilon_1}{2}\Bigr\},
\end{align*}
 \begin{align*}
     a_l =  
     \left\{
\begin{array}{ll}
\lambda^{(k)}_{p_{m}} - \frac{1}{2V}, & \text{$l = (k-1)\pi(\rho) + m$ for $1 \le k \le q-1$ and $1 \le m  \le \pi(\rho)$, }\\
\theta_{p_m} - \frac{1}{2J},  & \text{$l = (q-1)\pi(\rho) + m - \pi(\rho)$ for $\pi(\rho) < m < \pi(q)$, } \\ 
\theta_{p_m} - \frac{1}{2J},  & \text{$l = (q-1)\pi(\rho) + m - \pi(\rho) -1$ for $\pi(q) < m \le \pi(Q)$, } \\ 
\theta_{q} - \frac{\varepsilon_1}{2}, & \text{$l = (q-1)\pi(\rho) + \pi(Q) - \pi(\rho)$, }
\end{array}
\right.
 \end{align*}
 \begin{align*}
 b_l =  
     \left\{
\begin{array}{ll}
\lambda^{(k)}_{p_{m}} + \frac{1}{2V}, & \text{$l = (k-1)\pi(\rho) + m$ for $1 \le k \le q-1$ and $1 \le m  \le \pi(\rho)$, }\\
\theta_{p_m} + \frac{1}{2J},  & \text{$l = (q-1)\pi(\rho) + m - \pi(\rho)$ for $\pi(\rho) < m < \pi(q)$, } \\ 
\theta_{p_m} + \frac{1}{2J},  & \text{$l = (q-1)\pi(\rho) + m - \pi(\rho) -1$ for $\pi(q) < m \le \pi(Q)$, } \\ 
\theta_{q} + \frac{\varepsilon_1}{2}, & \text{$l = (q-1)\pi(\rho) + \pi(Q) - \pi(\rho)$, }
\end{array}
\right.
 \end{align*} 
 and 
  \begin{align*}
     \alpha_l =  
     \left\{
\begin{array}{ll}
\frac{1}{2\pi} \log p_m, & \text{$l = (k-1)\pi(\rho) + m$ for $1 \le k \le q-1$ and $1 \le m  \le \pi(\rho)$, }\\
\frac{1}{2\pi} \log p_m,  & \text{$l = (q-1)\pi(\rho) + m - \pi(\rho)$ for $\pi(\rho) < m < \pi(q)$, } \\ 
\frac{1}{2\pi} \log p_m,  & \text{$l = (q-1)\pi(\rho) + m - \pi(\rho) -1$ for $\pi(q) < m \le \pi(Q)$, } \\ 
\frac{1}{2\pi} \log q, & \text{$l = (q-1)\pi(\rho) + \pi(Q) - \pi(\rho)$.}
\end{array}
\right.
 \end{align*}
 Here, $(\theta_p)_{\rho < q < Q, p \ne q}$ are real numbers that we define later. 
 We put $n(\rho) = 1$ if $q \le \rho$ and $n(\rho) = 0$ if $q > \rho$, $L = (q-1)(\pi(\rho) - n(\rho)) + (\pi(Q) - \pi(\rho) - 1 + n(\rho)) + 1$, $U =U(\theta)$, $H > JQ$, $H_l = \frac{H}{\log Q}$ for $l = 1, \dots, L-1$, $H_L = \frac{H}{\varepsilon_1 J\log Q}$, and we set
\[
\alpha = ( \alpha_1, \alpha_2,  \dots, \alpha_L).
\]
We consider applying Lemma~\ref{Koksma} with these notations.
By the definition of $a_l, b_l$, we see
 \begin{align*}
     \prod_{l = 1}^{L} (b_l - a_l) = V^{-(q-1)(\pi(\rho) -  n(\rho))} J^{\pi(\rho) -\pi(Q) + 1 - n(\rho)} \varepsilon_1.
 \end{align*}
 Furthermore, the inequality $(1 + x)^y - 1 \le e^{xy} -1$ for $x, y \ge 0$, the fact that $L \le (q-1)\pi(Q)$, and Lemma~\ref{Du} lead to 
 \begin{align*}
     \prod_{l \le L} \left(1 + \frac{75}{H_l(b_l-a_l)} \right) -1
     &\le \left(1 + \frac{75J\log Q}{H} \right)^{L} -1 \\ 
     &\le \exp\left(\frac{75J\log Q}{H}L \right) -1 \\
     &\le \exp\left(\frac{75J\log Q}{H}(q-1)\pi(Q) \right) -1\\
     &\le \exp\left(\frac{82.05(q-1)JQ}{H} \right) -1 \\
     &\le 0.46,
 \end{align*}
 if \begin{equation} \label{217}
     H \ge 217(q-1)JQ. 
 \end{equation}
 
 Next, we consider 
 \[
 \sideset{}{'}{\sum}_{(h)}\left| \sum_{l \le L} \alpha_l h_l\right|^{-1} \prod_{l = 1}^{L} \gamma_{h_l, l}.
 \]
 By the definition of $\gamma_{h_l, l}$, we can observe that if $h_{L} = 0$, then 
 \[
    \gamma_{h_{L}, L} = 
\varepsilon_1 + \frac{75\varepsilon_1 J\log Q}{H}
  < 1.003\varepsilon_1 \le 1
 \]
 holds. 
 Moreover, 
 \[
 |h_{L}| \le \frac{H}{\varepsilon_1 J \log Q}(1 + \log L) \le \frac{1.1}{\varepsilon_1}H 
 \]
 and for $1 \le l < L$, $h_l \le 1.1H$. 
 Applying these estimates, and the estimate of \cite{Ga03}, we have 
 \begin{align} \label{sum}
     \sideset{}{'}{\sum}_{(h)}\left| \sum_{p \le Q} \alpha_l h_l \right|^{-1}  \prod_{l \le L} \gamma_{h_l, l} &\le 2 \prod_{\substack{p \le \rho \\ p \ne q}} \left(1 + 60 \sum_{1 \le h_p \le 1.1H} \frac{p^{h_p}}{h_p} \right)^{q-1} \prod_{\substack{\rho < p \le Q \\ p \ne q}} \left(1 + 60 \sum_{1 \le h_p \le 1.1H} \frac{p^{h_p}}{h_p} \right)\notag \\
     & \times \left(1 + 60 \sum_{1 \le h_{q} \le \frac{1.1}{\varepsilon_1}H} \frac{q^{h_{q}}}{h_{q}} \right) 
 \end{align}
 By the partial summation and the inequality for the exponential integral 
 \[
\mathrm{Ei}(x) := \int_{-\infty}^{x} \frac{e^u}{u}\, du < \frac{e^x}{x} \left(1 + \frac{1}{x} + \frac{3}{x^2} \right)
 \]
 for $x \ge e^{13}$, the below inequalities hold:
\begin{align*}
    \sum_{1 \le n \le x} \frac{p^n}{n} &\le \frac{p^{\frac{20}{11}x}}{x\log p} + \int_{1}^{x} \frac{p^{\frac{20}{11}u}}{u^2\log p}\, du \\ 
    &\le \frac{p^{\frac{20}{11}x}}{x\log p} + \frac{1}{\log p} \mathrm{Ei}\left({\frac{20}{11}x\log p}\right) \\ 
    &\le \frac{p^{\frac{20}{11}x}}{x\log p} + \frac{p^{\frac{20}{11}x}}{\frac{20}{11}x\log p} \left( \frac{1}{\log p} + \frac{1}{\frac{20}{11}x (\log p)^2} + \frac{3}{\left(\frac{20}{11}x\right)^2 (\log p)^3} \right) 
\end{align*}
if $\frac{20}{11}x\log p > e^{13}$. 
Note that the Cauchy principal value is taken for $\mathrm{Ei}(x)$. 
Therefore we have 
\begin{align*}
   \sum_{1 \le n \le 1.1H} \frac{p^{n}}{n} \le 1.8 \frac{p^{2H}}{1.1H\log p}
\end{align*}
for $p \ne q$ and 
\begin{align*}
    \sum_{1 \le n \le \frac{1.1}{\varepsilon_1}H} \frac{q^{n}}{n} \le 1.8 \varepsilon_1 \frac{q^{\frac{2H}{\varepsilon_1}}}{1.1H\log q}.
\end{align*}
Hence the left hand side of (\ref{sum}) is bounded by 
\begin{align*}
 &2\varepsilon_1\left(\frac{100}{H}\right)^{L} q^{2H\left(\frac{1}{\varepsilon_1} -(q-1) \right)} (\log q)^{q-2}\left(\prod_{p \le Q} \frac{p^{2H}}{\log p} \right)^{q-1} \\
 &=  2\varepsilon_1\left(\frac{100}{H}\right)^{L} q^{2H\left(\frac{1}{\varepsilon_1} -(q-1) \right)} (\log q)^{q-2}\frac{\exp\left( 2H (q-1)\sum_{p \le Q} \log p\right)} {\exp\left((q-1)\sum_{p \le Q} \log\log p \right)} . 
\end{align*}
Now, using inequalities $\sum_{p \le x} \log p \le 1.000081x$ for $x > 0$ (see \cite{Sc}), and $x/\log x < \pi(x)$ for $x \ge 17$ (see \cite[Corollary~1]{RS}), we have 
\begin{align*}
    \sideset{}{'}{\sum}_{(h)}\left| \sum_{p \le Q} \frac{h_p \log p}{2}  \right|^{-1} \prod_{p \le Q} \gamma_{h_p, p} \le 2\varepsilon_1
    e^{(q-1)(1.1HQ - 255585)}  q^{2H\left(\frac{1}{\varepsilon_1} -(q-1) \right)} (\log q)^{q-2}. 
\end{align*}
If 
\begin{align} \label{inequ3}
    &\exp\left(1.1(q-1)HQ + 2H\left( \frac{1}{\varepsilon_1} -(q-1)\right)  \log q + (q-2)\log\log q\right) \notag \\
     &\le T V^{-(q-1)(\pi(\rho) - n(\rho))} J^{\pi(\rho) -\pi(Q) + 1 - n(\rho)}
\end{align}
holds, then 
\begin{equation} \label{inequ4}
    2e^{(q-1)(1.1HQ - 255585)}  q^{2H\left(\frac{1}{\varepsilon_1} -(q-1) \right)} (\log q)^{q-2}
    \le 0.01T V^{-(q-1)(\pi(\rho) - n(\rho))} J^{\pi(\rho) -\pi(Q) + 1 - n(\rho)}
\end{equation} 
is true. Now, since $H > Q \ge 355991$ holds,  
\begin{equation} \label{inequ5}
1.11(q-1)HQ + 2H\left( \frac{1}{\varepsilon_1} - (q-1) 
     \right) \log q + (q-2)\log\log q\le \log T
\end{equation}
implies (\ref{inequ3}). 
Therefore if
\begin{equation*}
    V\left(\frac{Q}{\rho} \right)^{\frac{1}{2}}(q-1)Q \left(241(q-1)Q + 434\left(\frac{1}{\varepsilon_1} -(q-1)\right) \log q + 217(q-2)\log\log q \right) \le \log T
\end{equation*}
     holds, we can take $H$ satisfying (\ref{217}) and (\ref{inequ5}). 
     Then, we have
     \begin{equation} \label{meas}
         T_\alpha(U(\theta)) \ge 0.53\varepsilon_1 T V^{- (q-1)(\pi(\rho) - n(\rho))} J^{\pi(\rho) -\pi(Q) + 1 - n(\rho)}. 
     \end{equation}

On the other hand, let $\sum_{(j)}$ denote the summation over all $(j_p)_{\substack{\rho < p \le Q \\ p \ne q}}$ with $j_p \in \mathbb{N}$ and $j_p \le J$. 
We recall that $\chi_k$ is a Dirichlet character mod $q$
for $k= 1, \dots, q-1$. 
By the same way as \cite[Lemma~3.4]{NP}, for $k = 1, \dots, q-1$, we can show that 
\begin{align*}
    &\sum_{(j)} \left|\sum_{\rho < p \le Q} \log\left( 1 - \frac{\chi_k(p)e^{-2\pi i j_p/J}}{p^{s+ \frac{3}{4}}}\right) \right|^2 
    = \sum_{(j)} \left|\sum_{\rho < p \le Q} \sum_{n=1}^{\infty} \frac{\chi_k^n(p)e^{-2\pi i n j_p/J}}{np^{n\left(s + \frac{3}{4}\right)}} \right|^2 \\
    &\le 2\sum_{(j)} \left|\sum_{\rho < p \le Q}  \frac{\chi_k(p)e^{-2\pi i  j_p/J}}{p^{s + \frac{3}{4}}} \right|^2 + 2\sum_{(j)} \left|\sum_{\rho < p \le Q} \sum_{n=2}^{\infty} \frac{\chi_k^n(p)e^{-2\pi i n j_p/J}}{np^{n\left(s + \frac{3}{4}\right)}} \right|^2 \\
    &\le 2J^{\pi(Q) - \pi(\rho) -1 + n(\rho)} \sum_{\rho < p \le Q} \frac{1}{p^{2\left(\sigma + \frac{3}{4} \right)}} + 2 \sum_{(j)} \left( \sum_{\rho < p \le Q} \frac{1}{p^{2\left(\sigma + \frac{3}{4} \right)}} \right)^2 \\
    &\le J^{\pi(Q) - \pi(\rho) -1 + n(\rho)} \left(\frac{3-4r}{1-4r} \right)^2 \frac{4.7}{\rho^{\frac{1}{2} -2r}\log \rho}. 
\end{align*}
Therefore, when $p$ runs over all primes satisfying $\rho < p \le Q$ and $p \ne q$, 
we obtain more than $0.98\, J^{\pi(Q) - \pi(\rho) - 1 + n(\rho)}$ sets $U((\theta_p))$ with 
$\theta_p = j_p / J\ (p \ne q)$ such that
\begin{equation} \label{partial2}
 \left|
 \sum_{\rho < p \le Q} 
   \log\!\left(1 - \frac{\chi_k(p)e^{-2\pi i\theta_p}}{p^{s + 3/4}} \right)
 \right|
 \le 
 \frac{3 - 4r}{1 - 4r}\,
 \frac{\sqrt{50 \times 4.7}}{\rho^{1/4 - r}\sqrt{\log \rho}}.
\end{equation}
Let $A$ denote the union of these sets.  
Then the measure of those $\tau \in (T, 2T)$ for which $\tau \alpha \in A \mod{1}$ is greater than
\[
0.53\, \varepsilon_1 T\, V^{-(q-1)(\pi(\rho) - n(\rho))} 
   J^{\pi(\rho) - \pi(Q) + 1 - n(\rho)} 
   \times 0.98\, J^{\pi(Q) - \pi(\rho) - 1 + n(\rho)} 
 \;\ge\; \frac{1}{2}\, \varepsilon_1 T\, V^{-(q-1)\pi(\rho)}.
\]

Finally, by Lemma~\ref{appro2} and \eqref{partial2}, for such $\tau$, all $|s| \le r$, and all $1 \le k \le q-1$, we have
\begin{align*}
&\left|
\sum_{p \le Q} 
  \log\left(1 - \frac{\chi_k(p)}{p^{s + \frac{3}{4} + i\tau}} \right)
 -
 \sum_{p \le \rho}
  \log\left(1 - \frac{\chi_k(p) e^{-2\pi i\lambda_p^{(k)}}}{p^{s + \frac{3}{4}}} \right)
\right| \\
&\le 
\left|
\sum_{p \le \rho} 
  \log\left(1 - \frac{\chi_k(p)}{p^{s + \frac{3}{4} + i\tau}} \right)
 -
\sum_{p \le \rho}
  \log\left(1 - \frac{\chi_k(p)e^{-2\pi i\lambda_p^{(k)}}}{p^{s + \frac{3}{4}}} \right)
\right| \\
&\quad
+ 
\left|
\sum_{\rho < p \le Q}
  \log\left(1 - \frac{\chi_k(p)}{p^{s + \frac{3}{4} + i\tau}} \right)
 -
\sum_{\rho < p \le Q}
  \log\left(1 - \frac{\chi_k(p)e^{-2\pi i\theta_p}}{p^{s + \frac{3}{4}}} \right)
\right| \\
&\quad
+ 
\left|
\sum_{\rho < p \le Q} 
  \log\left(1 - \frac{\chi_k(p)e^{-2\pi i\theta_p}}{p^{s + \frac{3}{4}}} \right)
\right| \\
&\le 
188\, \frac{\rho^{\frac{1}{4} + r}}{V \log \rho}
\;+\;
\frac{3 - 4r}{1 - 4r}\,
\frac{16}{\rho^{\frac{1}{4} - r}\sqrt{\log \rho}}, 
\end{align*}
and
\[
    \left\| \tau \frac{\log q}{2\pi} - \theta_q \right\| 
    < \frac{\varepsilon_1}{2}.
\]

 \end{proof}
 
\section{Proof of main results}
In this section, we prove the main results. 
Combining Proposition~\ref{prop1}, Proposition~\ref{prop2} and Proposition~\ref{prop3}, we can show Theorem~\ref{main1} easily. 
Therefore, it is enough to prove Corollary~\ref{main2}. 
Actually, we prove the next Corollary instead of Corollary~\ref{main2}. 
\begin{corollary} \label{main3}
    Let $q$, $\chi_k$, $\theta_q$, $R$, $\beta$, $r$, $\varepsilon$, $\rho$ and $g_k(s)$ be the same as Corollary~\ref{main2} for $k=1, \dots, q-1$. 
    Let $Q$ and $T$ be positive numbers such that 
    \[
    Q^{\frac{1}{4}-r} = \max\left\{\frac{(1.02)^2(q-1)^2}{(0.25-r)^{10} \varepsilon^4 \sqrt{2\varepsilon_1}}e^{2\rho}, \exp\left(\left(\frac{1}{4} -r \right) q^8 \right) \right\}
    \]
    and 
    \begin{align*}
        \log T &\ge \max\left\{\log{\pi} + \frac{1.02(0.75+r)(Q-872)}{0.25-r}, \right. \\ 
     &\left. \quad V\left(\frac{Q}{\rho} \right)^{\frac{1}{2}}(q-1)Q \left(241(q-1)Q + 434\left(\frac{1}{\varepsilon_1} -(q-1)\right) \log q + 217(q-2)\log\log q\right) \right\}. 
    \end{align*}
    Then, the measure of $\tau \in [T, 2T]$, such that 
    \begin{align*}
         \max_{1 \le k \le q-1} \max_{|s| \le r} \left|\log{L\left(s + \frac{3}{4} + i\tau, \chi_k \right) - g_k(s)} \right|< \varepsilon
     \end{align*}
     and 
     \[
     \left\|\tau \frac{\log{q}}{2\pi} - \theta_q \right\| < \varepsilon_1
     \]
     is greater than $2\varepsilon_1 e^{-(q-1)\rho} T$. 
\end{corollary}

\begin{proof}
    Put $V = \sqrt{\frac{\rho}{\log \rho}}$ and replace $\varepsilon_1$ by $2\varepsilon_1$. 
    Since $\rho \ge 355991$, by Lemma~\ref{Du}, we have 
    \[
    \frac{T}{2} \left(\varepsilon_1 V^{-\pi(\rho)} - 1.02(q-1) \frac{\varphi^4(q)}{q^4}\frac{\log^2{Q}}{\varepsilon^2 (0.25-r)^5 Q^{0.25-r}} \right) \ge 2\varepsilon_1 e^{-(q-1)\rho}T. 
    \]
    Furthermore, under the condition, we obtain
    \begin{align*}
        &\frac{M}{\left(\frac{R}{r} -1 \right) \rho^{\delta \log{\frac{R}{r}}}} + \frac{3}{\rho^\alpha \log{\rho}} +
         \frac{3\log{\rho}}{\rho^{(1-4\delta)(\frac{3}{4}-r)}} + \frac{3}{\rho^{(1-4\delta)(\frac{1}{2}-2r)}\log{\rho}} + \frac{188\rho^{\frac{1}{4} + r}}{V\log \rho} + \frac{3-4r}{1-4r} \frac{16}{\rho^{\frac{1}{4}-r} \sqrt{\log \rho}} \\ 
         &< \frac{\varepsilon}{2}. 
    \end{align*}
    Therefore, Corollary~\ref{main2} can be proven. 
\end{proof}

In Theorem~\ref{main1}, Corollary~\ref{main2} and Corollary~\ref{main3}, we consider the hybrid joint universality theorem for the logarithmic of Dirichlet $L$-functions. 
Applying these results, we immediately obtain a lower bound of the hybrid joint universality for Dirichlet $L$-functions. 

\begin{corollary} \label{main4} 
    Let $q$, $\chi_k$, $\theta_q$, $R$, $\beta$, $r$, $g_k(s)$, $\varepsilon_1$ be the same as Corollary~\ref{main3} for $k=1, \dots, q-1$.
    We assume that $\varepsilon > 0$ satisfies $ \varepsilon_2 := \varepsilon/(2\max_{1 \le k \le q-1} \max_{s \le r}|e^{g_k(s)}|) \le 1$. 
    Changing $\varepsilon$ by $\varepsilon_2$, we define $\rho$, $Q$ and $T$ satisfying the condition of Corollary~\ref{main3}. 
    Then, the measure of $\tau \in [T, 2T]$ such that 
    \begin{align*}
         \max_{1 \le k \le q-1} \max_{|s| \le r} \left|L\left(s + \frac{3}{4} + i\tau, \chi_k \right) - e^{g_k(s)} \right|< \varepsilon
     \end{align*}
     and 
     \[
     \left\|\tau \frac{\log{q}}{2\pi} - \theta_q \right\| < \varepsilon_1
     \]
    is not less than $2\varepsilon_1e^{-(q-1)\rho}T$. 
\end{corollary}

\section{Application for Hurwitz zeta-functions} \label{Hurwitz}

The Hurwitz zeta function is defined by
\[
\zeta(s, \alpha) = \sum_{n=0}^\infty (n + \alpha)^{-s}, \quad \sigma > 1,
\]
where the parameter $\alpha \in (0, 1]$.
The universality theorem for $\zeta(s, \alpha)$ with transcendental or rational $\alpha$
was proved independently by Bagchi~\cite{Ba} and Gonek~\cite{Gon}.
It is known that the universality theorem for Hurwitz zeta-functions with a rational parameter can be proven by the hybrid joint universality theorem for Dirichlet $L$-functions. 
Using this fact, we can estimate a lower bound of universality for Hurwitz zeta-functions with a rational parameter in which the denominator is a prime. 

Let $q$ be an odd prime number, $p$ be a positive integer with $(p, q) = 1$ and $r$, $g(s)$ be the same as Corollary~\ref{main2}. 
We choose a positive $C$ with 
    \[
    \max_{|s| \le r} |g(s)| < \frac{C}{q^r}
    \]
and put 
    \begin{align} \label{divide}
    g(s, \chi_k) = 
    \begin{cases}
    \chi_k(p)(g(s)q^{-s} + C),  & k = 1, \dots, \frac{q-1}{2}, \\
    \chi_k(p)(g(s)q^{-s} - C),  & k = \frac{q-1}{2} + 1, \dots, q-1, 
    \end{cases}
    \end{align}
where $\chi_1, \dots, \chi_{q-1}$ are distinct Dirichlet characters mod $q$. 
Then, each $g(s, \chi_k)$ is non-vanishing on $|s| \le r$. 
Let $\varepsilon > 0$ be such that 
\begin{align} \label{epsilon}
 \frac{\varepsilon}{4q^r \max_{1 \le k \le q-1} \max_{|s| \le r} |g(s, \chi_k)|} \le 1. 
\end{align}
Then, applying Corollary~\ref{main4}, we have the following Corollary. 
\begin{corollary} \label{main5}
    Let $q$ be an odd prime number, $p$ be a positive integer with $(p, q) = 1$ and $R$, $\beta$, $r$, $g(s)$ be the same as Corollary~\ref{main2}.
    We choose $\varepsilon$ satisfying (\ref{epsilon}) and put 
    \[
    \varepsilon_1 := \frac{\varepsilon}{4\pi q^r \max_{1 \le k \le q-1} \max_{|s| \le r} |g(s, \chi_k)|}\ \text{and}\ 
    \varepsilon_2 := \frac{\varepsilon}{4q^r \max_{1 \le k \le q-1} \max_{|s| \le r} |g(s, \chi_k)|}. 
    \]
    Let $\rho$, $Q$, $T$ be positive numbers with
    \[
    \rho \ge \left(\frac{2a(r)}{\varepsilon} \right)^{\frac{2}{(1-4\delta)(\frac{1}{4}-r)}},\ \frac{\rho^{\beta}}{5e\delta^3 \log^4\rho} \ge \max_{1 \le k \le q-1}\max_{|s| \le 0.06} |g(s, \chi_k)| + 6, 
    \]
    \[
    Q^{\frac{1}{4}-r} = \max\left\{\frac{(1.02)^2(q-1)^2}{(0.25-r)^{10} \varepsilon_2^4 \sqrt{2\varepsilon_1}}e^{2\rho}, \exp\left(\left(\frac{1}{4} -r \right) q^8 \right) \right\}, 
    \]
    and 
      \begin{align*}
        \log T &\ge \max\left\{\log{\pi} + \frac{1.02(0.75+r)(Q-872)}{0.25-r}, \right. \\ 
     &\left. \quad V\left(\frac{Q}{\rho} \right)^{\frac{1}{2}}(q-1)Q \left(241(q-1)Q + 434\left(\frac{1}{\varepsilon_1} -(q-1)\right) \log q + 217(q-2)\log\log q\right) \right\}. 
    \end{align*}
    where $a(r)$ is defined in Corollary~\ref{main2} and $g(s, \chi_k)$ is defined by (\ref{divide}). 
    Then, we have
    \begin{equation*}
        \meas \left\{\tau \in [T, 2T] :  \sup_{|s| \le r} \left|\zeta\left(s + \frac{3}{4} + i\tau, \frac{p}{q} \right) -g(s) \right| < \varepsilon \right\} \ge 2\varepsilon_1e^{-\rho}T. 
    \end{equation*}
    \end{corollary}

On the other hand, Gonek~\cite{Gon} conjectured that universality also holds for the Hurwitz zeta function
with an algebraic irrational parameter.
However, the universality theorem for Hurwitz zeta functions with algebraic irrational parameters
has not yet been completely established.
Recently, Sourmelidis and Steuding~\cite{SS} proved a weak form of universality for Hurwitz zeta functions with algebraic irrational parameters, depending on some auxiliary parameters. 
Later, Mine~\cite{Mi} showed that universality for $\zeta(s, \alpha)$
holds for all but finitely many $\alpha$ in
\[
\mathcal{A}_\rho(c)
:= \{\, 0 < \alpha < 1 : \text{$\alpha$ is an algebraic irrational number satisfying } |\alpha - c| \le \rho \,\},
\]
where $0 < c < 1$ and $0 < \rho \le \min\{c, 1-c\}$.
 
We consider a different approach those in from previous studies~\cite{SS} and \cite{Mi}, which is based on a lower bound of the lower density. 
We introduce the definition and property of normal families. 
Write $H(\Omega)$ for the set of all holomorphic functions on $\Omega.$

\begin{definition} [{cf. \cite[Definition~14.5]{Ru}}]
    Suppose $\mathcal{F} \subset H(\Omega)$, for some region $\Omega$. 
    We call $\mathcal{F}$ a normal family if every sequence of members of $\mathcal{F}$ contains a subsequence which converges uniformly on compact subsets of $\Omega$. 
    The limit function is not required to belong to $\mathcal{F}$. 
\end{definition}

\begin{theorem} [{cf. \cite[Theorem~14.6]{Ru}}] \label{uni bounded}
    Suppose $\mathcal{F} \subset H(\Omega)$ and $\mathcal{F}$ is uniformly bounded on each compact subset of the region $\Omega$. 
    Then $\mathcal{F}$ is a normal family. 
\end{theorem}

\begin{lemma} \label{normal}
    Denote $D := \{s : 1/2 < \sigma < 1 \}$.  
    Then, $\mathcal{F}_{\alpha_0} := \{\zeta(s, \alpha) : 0 < \alpha_0 \le \alpha \le 1 \} \subset H(D)$ is a normal family for fixed $\alpha_0 \in (0, 1)$.
\end{lemma}

\begin{proof}
    Assume $K$ is a compact set of $D$.
    Then, by the well-known result (cf.\cite[Lemma~I.4.3]{KV}), for any $\zeta(s, \alpha) \in \mathcal{F}_{\alpha_0}$, we have
    \begin{align*}
    \zeta(s, \alpha) 
    &= \alpha^{-s} + (1 + \alpha)^{-s} + \frac{1}{s-1} \left(\alpha + \frac{3}{2} \right)^{1-s} + s \int_{\frac{3}{2}}^{\infty} \frac{u - \lfloor u \rfloor - \frac{1}{2}}{(u + \alpha)^{s+1}}\, du \\
    &\ll \alpha^{-\sigma} + (1 + \alpha)^{-\sigma} + \frac{1}{|s-1|}\left(\alpha + \frac{3}{2} \right)^{1-\sigma} + \frac{|s|}{\sigma} \left(\alpha + \frac{3}{2} \right)^{-\sigma} \\
    &\ll_{K, \alpha_0} 1 . 
    \end{align*}
    Therefore, $\mathcal{F}_{\alpha_0}$ is uniformly bounded on $K$. 
    By Theorem~\ref{uni bounded}, we see that $\mathcal{F}_{\alpha_0}$ is a normal family in $H(D)$. 
\end{proof}

Applying this lemma, we can prove the following. 

\begin{theorem} \label{lower density}
Let $K$ be a compact set in the strip $1/2 < \sigma < 1$ with connected complement, and let $g(s)$ be a continuous function on $K$ that is analytic in the interior of $K$. 
Suppose $\alpha \in (0, 1)$ is an algebraic irrational number and $\{\alpha_n\} \subset (0, 1] \cap \mathbb{Q}$ converges to $\alpha$. 
If
\[
\liminf_{T \to \infty} \limsup_{n \to \infty} \frac{1}{T} \meas\left\{\tau \in [T, 2T] : \sup_{s \in K } |\zeta(s+i\tau, \alpha_n) - g(s)| < \varepsilon \right\} \ge c(\varepsilon) > 0
\]
holds for any enough small $\varepsilon > 0$, then 
\[
\liminf_{T \to \infty} \frac{1}{T} \meas \left\{\tau \in [T, 2T] :  \sup_{s \in K} |\zeta(s + i\tau, \alpha) -g(s)| < \varepsilon \right\} > 0
\]
holds. 
\end{theorem}

\begin{proof} 

    We put 
    \[
    E_n := \left\{\tau \ge 0 : \sup_{s \in K} |\zeta(s + i\tau, \alpha_n) - g(s) | < \frac{\varepsilon}{2} \right\}, 
    \]
    \[
    E := \left\{\tau \ge 0 : \sup_{s \in K} |\zeta(s + i\tau, \alpha) - g(s) | < \varepsilon \right\}. 
    \]
    First, we prove 
    \[
    \limsup_{n \to \infty} E_n \subset E. 
    \]
    Assume $\tau \in \limsup_{n \to \infty} E_n$. 
    By the definition of $\limsup_{n \to \infty} E_n$, we find a subsequence $n_k$ such that 
    \[
    \sup_{s \in K} |\zeta(s + i\tau, \alpha_{n_k}) - g(s)| < \frac{\varepsilon}{2} 
    \]
    holds for sufficiently large $k$. 
    Here, $\{\zeta(s, \alpha_{n_k}) : k \in \mathbb{N}\}$ is a normal family, and $K + i\tau$ is a compact set.  
    Therefore, by (\ref{normal}), there exists a subsequence $k_j$ such that 
    \[
    \sup_{s \in K} |\zeta(s + i\tau, \alpha_{n_{k_j}}) - \zeta(s + i\tau, \alpha) | < \frac{\varepsilon}{2}. 
    \]
    Combining these two inequalities, we get 
    \[
    \sup_{s \in K} |\zeta(s + i\tau, \alpha) - g(s)| < \varepsilon,
    \]
    i.e. $\tau \in E$. 
    Finally, we estimate the lower density of \[
\liminf_{T \to \infty} \frac{1}{T} \meas \left\{\tau \in [T, 2T] :  \sup_{s \in K} |\zeta(s + i\tau, \alpha) -f(s)| < \varepsilon \right\}. 
\]
We write that $I_A$ is the indicator function of a set $A$. 
Then, using the fact that 
\[
I_{\limsup_{n \to \infty} E_n} = \limsup_{n \to \infty} I_{E_n},
\] and Fatou's lemma, we have 
\begin{align*}
    &\frac{1}{T}\meas \left\{\tau \in [T, 2T] :  \sup_{s \in K} |\zeta(s + i\tau, \alpha) -g(s)| < \varepsilon \right\} = \frac{1}{T} \int_{T}^{2T} I_{E}(\tau)\, d\tau \\ 
    &\ge \frac{1}{T} \int_{T}^{2T} I_{\limsup_{n \to \infty} E_n} (\tau)\, d\tau \\
    &=  \frac{1}{T} \int_{T}^{2T} \limsup_{n \to \infty} I_{E_n} (\tau)\, d\tau \\
    &\ge \limsup_{n \to \infty} \frac{1}{T} \int_{T}^{2T} I_{E_n}(\tau)\, d\tau. 
    \end{align*}
Taking $\liminf_{T \to \infty}$, we conclude 
\begin{align*}
    &\liminf_{T \to \infty} \frac{1}{T} \meas \left\{\tau \in [T, 2T] :  \sup_{s \in K} |\zeta(s + i\tau, \alpha) -g(s)| < \varepsilon \right\} \\
    &\ge \liminf_{T \to \infty} \limsup_{n \to \infty} \frac{1}{T} \meas\left\{\tau \in [T, 2T] : \sup_{s \in K } |\zeta(s+i\tau, \alpha_n) - g(s)| < \frac{\varepsilon}{2} \right\} \\
    &\ge c(\varepsilon/2) > 0
\end{align*}
 by the assumption. 
\end{proof}

In other words, if we can estimate a good lower bound of the lower density of universality for $\zeta(s, \alpha)$ for all $\alpha \in (0, 1] \cap \mathbb{Q}$, we can prove the universality theorem for $\zeta(s, \alpha)$, which $\alpha$ is an algebraic irrational number.  
Unfortunately, our lower bound given by Corollary~\ref{main5} does not work, because in Corollary~\ref{main5}, our lower bound depends on the parameters.  

\subsection*{Acknowledgments}
The author would like to thank Professor Kohji Matsumoto and Professor Ram\= unas Garunk\v stis for their helpful comments, Professor Yuta Suzuki for teaching Lemma~\ref{appro1}, and Dr. Yuya Kanado for his valuable comments. 
This work was financially supported by JSPS Research Fellow (Grant Number:25KJ1412).

\begin{flushleft}
{\footnotesize
{\sc
Graduate School of Mathematics, Nagoya University, Chikusa-ku, Nagoya 464-8602, Japan.
}\\
{\it E-mail address}: {\tt m21029d@math.nagoya-u.ac.jp}
}
\end{flushleft}

\end{document}